\documentclass[12pt]{article}

\RequirePackage[OT1]{fontenc}
\RequirePackage{amsthm,amsmath}
\RequirePackage[numbers]{natbib}
\RequirePackage[colorlinks,citecolor=blue,urlcolor=blue]{hyperref}

\usepackage{graphicx}
\usepackage[utf8]{inputenc}

\usepackage{subfigure}
\usepackage{float}
\usepackage{amsfonts}
\usepackage{enumitem}
\usepackage{amssymb}
\usepackage{color}
\usepackage{eucal}
\newtheorem{theorem}{Theorem}[section]
\newtheorem{cor}[theorem]{Corollary}
\newtheorem{conj}[theorem]{Conjecture}
\newtheorem{prop}[theorem]{Proposition}

\newtheorem{lemma}[theorem]{Lemma}

\newtheorem{manuallemmainner}{Lemma}
\newenvironment{manuallemma}[1]{%
  \renewcommand\themanuallemmainner{#1}%
  \manuallemmainner
}{\endmanuallemmainner}

\newtheorem{remark}[theorem]{Remark}

\newcommand{\ba}{\begin{eqnarray*}}
\newcommand{\ea}{\end{eqnarray*}}
\newcommand{\bq}{\begin{eqnarray}}
\newcommand{\eq}{\end{eqnarray}}

\begin{document}

\title{Genetic composition of an exponentially growing cell population}

\author{David Cheek\footnote{Program for Evolutionary Dynamics, Harvard University, dmcheek7@gmail.com}~  and Tibor Antal\footnote{School of Mathematics, University of Edinburgh, tibor.antal@ed.ac.uk}}

\date{}



\maketitle

\begin{abstract}
We study a simple model of DNA evolution in a growing population of cells. Each cell contains a nucleotide sequence which randomly mutates at cell division. Cells divide according to a branching process. Following typical parameter values in bacteria and cancer cell populations, we take the mutation rate to zero and the final number of cells to infinity. We prove that almost every site (entry of the nucleotide sequence) is mutated in only a finite number of cells, and these numbers are independent across sites. However independence breaks down for the rare sites which are mutated in a positive fraction of the population. The model is free from the popular but disputed infinite sites assumption. Violations of the infinite sites assumption are widespread while their impact on mutation frequencies is negligible at the scale of population fractions. Some results are generalised to allow for cell death, selection, and site-specific mutation rates. For illustration we estimate mutation rates in a lung adenocarcinoma.

\end{abstract}
\pagebreak
%
%

\tableofcontents

\section{Introduction}\label{introduction}
A population of dividing cells with a mutating DNA sequence is ubiquitous in biology. We study a simple model of this process. Starting with one cell, cells divide and die according to a supercritical branching process. As for DNA, we loosely follow classic models from phylogenetics \cite{jcm,tavarep}. Each cell contains a sequence of the nucleotides A, C, G, and T, and each site (entry of the sequence) can mutate independently at cell division. We are interested in the sequence distribution when the population reaches many cells. 

Let's discuss a specific motivation.  In recent years, cancer genetic data has been made available in great quantities. One especially common type of data consists of mutation frequencies in individual tumours. These data take the form of a vector $(x_i)_{i\in\mathcal{S}}$, where $i\in\mathcal{S}$ denotes genetic sites and $x_i$ is the frequency of cells which are mutated at site $i$. To make sense of such data in terms of tumour evolution, simple mathematical models can be helpful.

Some important works on the topic are \cite{gs,ib,matds,tavarerecent}. They consider branching process and deterministic models of tumour evolution. They compare theory with data, estimating evolutionary parameters such as mutation rates. A central feature of their theory, and of countless other works, is the so-called infinite sites assumption (ISA). The ISA states that no genetic site can mutate more than once in a tumour's lifetime. The assumption's simplicity drives its popularity. However recent statistical analysis of single cell sequencing data \cite{isr} shows ``widespread violations of the ISA in human cancers".

For a `non-ISA' model of a growing population of cells, there is in fact a famous example. Luria and Delbr{\"u}ck \cite{ld} modelled recurrent mutations in an exponentially growing bacterial population. Subsequent works \cite{lc,dk,hy,ka,kl} (and others) adapted Luria and Delbr{\"u}ck's model to branching processes and calculated mutation frequencies. These works describe only two genetic states, mutated or not mutated, effectively restricting attention to a single genetic site. In \cite{ca} we offered an account of one such model, proving limit theorems for mutation times, clone sizes, and mutation frequencies. We then briefly studied an extension to a sequence of genetic sites. Now we offer a self-contained sequel to \cite{ca}, slightly adapting the model, and aiming for a deeper understanding of the sequence distribution.

In \cite{ca} we studied several parameter regimes. In the present work by contrast, we study only one parameter regime which is the most biologically relevant. We take the final number of cells to infinity and the mutation rate to zero with their product finite. This limit is relevant because a detected tumour has around $10^9$ cells while the mutation rate per site per cell division is around $10^{-9}$ \cite{jea}. This limit is also standard in Luria-Delbr{\"u}ck-type models of bacteria. 

Now we introduce our main results. The number of cells mutated at a given site (mutations are defined relative to the initial cell) converges to the Luria-Delbr{\"u}ck distribution. This recovers a well-known result of single site models \cite{lc,hy,kl,ka,ca}. So a site is mutated in only a finite number of cells, standing in contrast to the infinite total number of cells. Going beyond \cite{lc,hy,kl,ka,ca}, we also study the rare event that a site is mutated in a positive fraction of cells. We show that, when appropriately scaled, this fraction of cells follows a power-law distribution.

Across sites, mutation frequencies are asymptotically independent. The independence leads to a many-sites law of large numbers. Specifically, the site frequency spectrum (empirical measure of mutation frequencies) converges to a deterministic measure concentrated at finite cell numbers. At positive fractions of cells, away from the mass concentration, independence breaks down and the site frequency spectrum converges to a Cox process. These results go beyond \cite{durr,gs,ib,tavarerecent,ca} who only give the expected site frequency spectrum, so our work contributes an appreciation of randomness.

Our results are not all at the same level of generality. For sites mutated in a positive fraction of cells, results are proven for a zero death rate and homogeneous division and mutation rates. For sites mutated in a finite number of cells, results are proven for sequence-dependent death, division, and mutation rates.


We also assess the infinite sites assumption's validity. Our results say that for typical parameter values, the number of sites to violate the ISA is at least millions, or even billions, in a single tumour. Thus our work agrees with \cite{isr}'s statistical analysis of single cell sequencing data which says that ISA violations are widespread. It should be emphasised however that ISA violations do not neccessarily invalidate the ISA. One of our results says that ISA violations do not impact mutation frequencies viewed at the scale of population fractions. Bulk sequencing data, which is the majority of cancer genetic data \cite{tavarerecent}, views mutation frequencies at the scale of population fractions. Therefore our work endorses analyses of bulk sequencing data which are reliant on the ISA, such as \cite{gs,ib,matds,tavarerecent}.

Before commencing the paper, let's note that there are a wealth of other works on mutations in branching processes. Especially common are infinite alleles models, for example \cite{griffithspakes,lambertchampagnat,Olly,duchamps}, where each individual in the population has an allele which can mutate to alleles never before seen in the population. In an infinite alleles model, a mutation always deletes an individual's ancestral genetic information. In an infinite sites model on the other hand, a mutation never deletes ancestral genetic information; mutations simply accumulate. The DNA sequence model which we study sits between those extremes.

The paper is structured as follows. In Section \ref{sec:model}, we introduce the model in its simplest form. In Section \ref{prelim}, we give notation and preliminary ideas. In Section \ref{mutationfrequencies}, we present the paper's main results. In Section \ref{gmrcs}, we give generalisations and open questions. In Section \ref{proofsecfinite}, we prove results on sites mutated in a finite number of cells. In Section \ref{sec:mpfps}, we prove results on sites mutated in a positive fraction of cells. In Section \ref{infinitesitessection}, we discuss the infinite sites assumption's validity. In Section \ref{sec:data}, we consider data from a lung adenocarcinoma and estimate mutation rates.

\section{Model}\label{sec:model}
Here the model is stated in its simplest form. It comprises two parts.
\begin{enumerate}\item Population dynamics: Starting with one cell, cells divide according to the Yule process. That is, cells divide independently at constant rate.
\item
Genetic information: The set of nucleotides is $\mathcal{N}=\{A,C,G,T\}$. The set of genetic sites is some finite set $\mathcal{S}$. The set of genomes (or DNA sequences) is $\mathcal{G}=\mathcal{N}^\mathcal{S}$. Each cell has a genome, i.e. is assigned an element of $\mathcal{G}$. Suppose that a cell with genome $(v_i)_{i\in\mathcal{S}}\in\mathcal{G}$ divides to give daughter cells with genomes $(V_i^{(1)})_{i\in\mathcal{S}}$ and $(V_i^{(2)})_{i\in\mathcal{S}}$. Conditional on $(v_i)$, the $V_i^{(r)}$ are independent over $i\in\mathcal{S}$ and $r\in\{1,2\}$, and
$$\mathbb{P}\left[V_i^{(r)}=\psi|(v_i)\right]=\begin{cases}\mu/3,\quad &\psi\not=v_i;\\1-\mu,\quad& \psi=v_i.
\end{cases}
$$
It is also assumed that mutations occur independently for different cell divisions.
\end{enumerate}

The model is generalised to cell death, selection, and nucleotide/site-specific mutation rates in Section \ref{gmrcs}. 
\section{Preliminaries}\label{prelim}
\subsection{Luria-Delbr{\"u}ck distribution}
Let $(Y_k)_{k\in\mathbb{N}}$ be an i.i.d. sequence of random variables with 
\[
\mathbb{P}[Y_1=j]=\frac{1}{j(j+1)}
\]
for $j\in\mathbb{N}$. Let $K$ be an independent Poisson random variable with mean $c$. The \textit{Luria-Delbr{\"u}ck distribution} with parameter $c$ is defined as the distribution of
\bq\label{ldpr}
B=\sum_{k=1}^KY_k.
\eq
It is commonly seen in its generating function form (e.g. \cite{lc,qz})
\bq\label{ldgfr}
\mathbb{E}z^B=(1-z)^{c(z^{-1}-1)}.
\eq
The connection between (\ref{ldpr}) and (\ref{ldgfr}) is made explicit in \cite{ca} for example.
Although the distribution is named after Luria and Delbr{\"u}ck (due to their groundbreaking work \cite{ld}), it was derived by Lea and Coulson \cite{lc}. See \cite{qz} for a historical review.

The Luria-Delbr{\"u}ck distribution's power-law tail was derived in \cite{pakestail}.
\begin{lemma}\label{lddplt}$\lim_{m\rightarrow\infty}m\mathbb{P}\left[B\geq m\right]=c.$
\end{lemma}
\subsection{Yule tree}\label{sec:Yuletreesec}
The set of all cells to ever exist, following standard notation, is
\[
\mathcal{T}=\cup_{l=0}^\infty\{0,1\}^l.
\]
A partial ordering, $\prec$, is defined on $\mathcal{T}$. For $x,y\in\mathcal{T}$, $x\prec y$ means that cell $y$ is a descendant of cell $x$. That is, $x\prec y$ if
\begin{enumerate}
\item
there are $l_1,l_2\in\mathbb{N}_0$ with $l_1<l_2$ and $x\in\{0,1\}^{l_1},y\in\{0,1\}^{l_2}$; and
\item
the first $l_1$ entries of $y$ agree with the entries of $x$.
\end{enumerate}
Note that $\emptyset\in\mathcal{T}$ and that $\emptyset\prec x$ for any $x\in\mathcal{T}\backslash\{\emptyset\}$. So $\emptyset$ is the initial cell from which all other cells descend. For further notation, write $x\preceq y$ if $x\prec y$ or $x=y$. Also, write $x0$ and $x1$ for the daughters of $x\in\mathcal{T}$; precisely, if $x\in\mathcal{T}$ and $j\in\{0,1\}$, then $xj$ is the element of $\{0,1\}^{l+1}$ whose first $l$ entries are the entries of $x$ and whose last entry is $j$.

Let $(A_x)_{x\in\mathcal{T}}$ be a family of i.i.d. exponentially distributed random variables with mean $1$. $A_x$ is the lifetime of cell $x$. The cells alive at time $t$ are
\[
\mathcal{T}_t:=\left\{x\in\mathcal{T}:\sum_{y\prec x}A_y\leq t<\sum_{y\preceq x}A_y\right\}.
\]

The proportion of cells alive at time $t$ which are descendants of cell $x$ (including $x$) is
\[
P_{x,t}:=\frac{|\{y\in\mathcal{T}_t:x\preceq y\}|}{|\mathcal{T}_t|}.
\]
\begin{lemma}\label{btst}
For each $x\in\mathcal{T}$,
\[
\lim_{t\rightarrow\infty}P_{x,t}=P_x:=\prod_{\emptyset\prec y\preceq x}U_y
\]
almost surely, where 
\begin{enumerate}
\item
the $U_y$ are uniformly distributed on $[0,1]$;
\item
for any $y\in\mathcal{T}$, $U_{y0}+U_{y1}=1$;
\item
$(U_{y0})_{y\in\mathcal{T}}$ is an independent family.
\end{enumerate}
\end{lemma}
Lemma \ref{btst} will be proved in Section \ref{sec:mpfps}.
\subsection{Mutation frequency notation}\label{sec:mfn}
When DNA is taken from a tumour, the tumour's age is unknown, but one may have a rough idea of its size. Therefore we are interested in the cells' genetic state at the random time
\[
\sigma_n=\min\{t\geq0:|\mathcal{T}_t|=n\},
\]
when the total number of cells reaches some given $n\in\mathbb{N}$.

Write $V^\mu(x)=(V_i^\mu(x))_{i\in\mathcal{S}}\in\mathcal{G}$ for the genome of cell $x\in\mathcal{T}$ (where $\mu$ is the mutation rate). So $(V^\mu(x))_{x\in\mathcal{T}}$ is a Markov-process indexed by $\mathcal{T}$ with transition rates given in Section \ref{sec:model}. Write $V^\mu(\emptyset)=(u_i)_{i\in\mathcal{S}}$ for the initial cell's genome. A genetic site is said to be mutated if its nucleotide differs from that of the initial cell. Note that, according to this definition, a site which mutates and then sees a reverse mutation to its initial state is not considered to be mutated. Write
\bq\label{keynotation}
B^{n,\mu}_i=\left|\{x\in\mathcal{T}_{\sigma_n}:V^\mu_i(x)\not=u_i\}\right|
\eq
for the number of cells which are mutated at site $i\in\mathcal{S}$ when there are $n$ cells in total. The quantity (\ref{keynotation}), and its joint distribution over $\mathcal{S}$, is the key object of our study.
\subsection{Parameter regime}
The number of cells in a detected tumour may be in the region of $n=10^9$, whereas the mutation rate is in the region of $\mu=10^{-9}$ \cite{jea}.  The human genome's length is around $|\mathcal{S}|=3\times10^9$. Very roughly,
\[
n\approx\mu^{-1}\approx|\mathcal{S}|.
\]
Therefore we study the limits:
\begin{itemize}
\item
$n\rightarrow\infty$, $\mu\rightarrow0$, $n\mu\rightarrow\theta<\infty$;
\item
$n\rightarrow\infty$, $\mu\rightarrow0$, $n\mu\rightarrow\theta<\infty$, $|\mathcal{S}|\rightarrow\infty$ (sometimes with $|\mathcal{S}|\mu\rightarrow\eta<\infty$). 
\end{itemize}
\begin{remark}Taking the number of sites to infinity is not to be confused with the infinite sites assumption.\end{remark}
\section{Main results}\label{mutationfrequencies}
The first result shows that sites are typically mutated in only a finite number of cells, and that these numbers are independent across sites.
\begin{theorem}\label{idlda} As $n\rightarrow\infty$ and $n\mu\rightarrow\theta\in[0,\infty)$,
\ba(B_i^{n,\mu})_{i\in\mathcal{S}}\rightarrow(B_i)_{i\in\mathcal{S}}
\ea
in distribution, where the $B_i$ are i.i.d. and have Luria-Delbr{\"u}ck distribution with parameter $2\theta$.
\end{theorem}
\begin{remark}Taking $|\mathcal{S}|=1$, Theorem \ref{idlda} recovers results of single site models \cite{lc,hy,ka,kl,ca}.
\end{remark}
The site frequency spectrum is a popular summary statistic of genetic data. It is defined as the empirical measure of mutation frequencies:
\[
\sum_{i\in\mathcal{S}}\delta_{B_i^{n,\mu}}.
\]
The site frequency spectrum sees a law of large numbers. 
\begin{theorem}\label{llnldd}As $n\rightarrow\infty$, $n\mu\rightarrow\theta\in[0,\infty)$, and $|\mathcal{S}|\rightarrow\infty$,
\ba
\frac{1}{|\mathcal{S}|}\sum_{i\in\mathcal{S}}\delta_{B_i^{n,\mu}}\rightarrow \Lambda
\ea
in probability, where $\Lambda$ is the Luria-Delbr{\"u}ck distribution with parameter $2\theta$. Convergence is on the space of probability measures on the non-negative integers equipped with the topology of weak convergence.
\end{theorem}
Theorems \ref{idlda} and \ref{llnldd} teach us that almost every site is mutated in only a finite number of cells. What about the rare sites which are mutated in a positive fraction of cells? Heuristically, the Luria-Delbr{\"u}ck distribution's tail gives the probability that site $i$ is mutated in at least fraction $a$ of cells:
\bq\mathbb{P}[n^{-1}B^{n,\mu}_i>a]&\approx&\mathbb{P}[B_i\in (na,n)]\label{hwapp}\\
&\approx&2\mu(a^{-1}-1).\label{lddta}
\eq
Approximation (\ref{hwapp}) is a hand-waving consequence of Theorem \ref{idlda}. Approximation (\ref{lddta}) is due to Lemma \ref{lddplt}. The next result offers rigour.
\begin{theorem}\label{sslmf}Let $i\in\mathcal{S}$ and $a\in(0,1)$. As $n\rightarrow\infty$ and $n\mu\rightarrow\theta\in[0,\infty)$,
\ba
\mu^{-1}\mathbb{P}[n^{-1}B^{n,\mu}_i>a]\rightarrow2(a^{-1}-1).
\ea
\end{theorem}
\begin{figure}
\begin{center}
\includegraphics{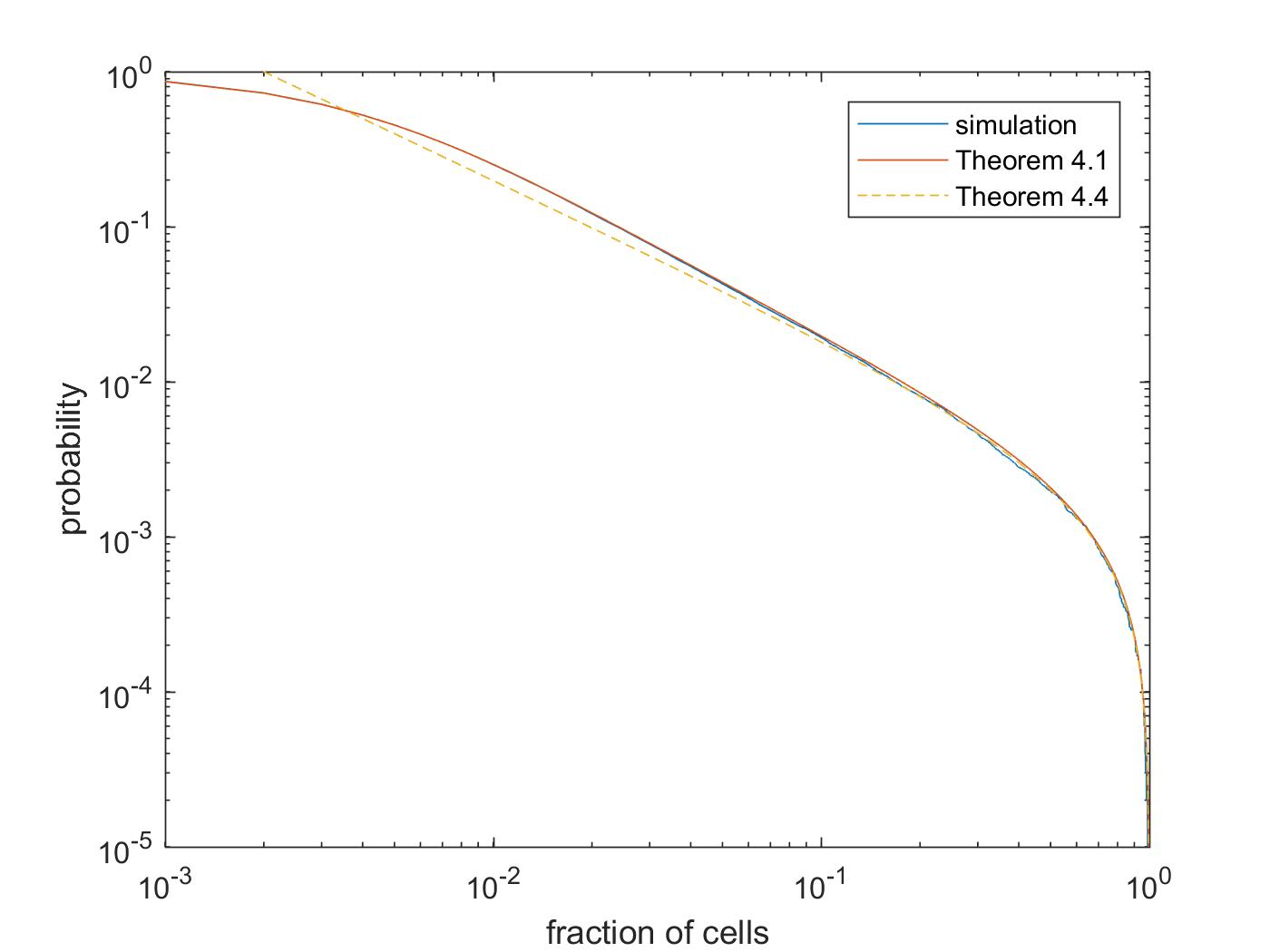}
	\caption{The number of mutant cells with respect to a single site is simulated $10^5$ times. The parameters are $\mu=10^{-3}$ and $n=10^3$. The plot compares $\mathbb{P}[n^{-1}B_1^{n,\mu}\in(a,1)]$ (simulation), $\mathbb{P}[n^{-1}B_1\in(a,1)]$ (Theorem \ref{idlda}), and $2\mu(a^{-1}-1)$ (Theorem \ref{sslmf}),  for $a\in(0,1)$. Simulation and Theorem \ref{idlda} appear indistinguishable.}\label{ssnmf}
\end{center}
\end{figure}
Theorem \ref{sslmf} and linearity of expectation yield the mean site frequency spectrum at positive fractions of the population.
\begin{cor}\label{lmfsfsm}Let $a\in(0,1)$. As $n\rightarrow\infty$, $n\mu\rightarrow\theta\in[0,\infty)$, and $|\mathcal{S}|\mu\rightarrow\eta\in[0,\infty)$,
\[
\mathbb{E}\sum_{i\in\mathcal{S}}\delta_{n^{-1}B_i^{n,\mu}}(a,1)\rightarrow 2\eta(a^{-1}-1).
\]
\end{cor}

The next result gives the distribution of the site frequency spectrum at positive fractions of the population.

\begin{theorem}\label{dslfm}As $n\rightarrow\infty$, $n\mu\rightarrow\theta\in[0,\infty)$, and $|\mathcal{S}|\mu\rightarrow\eta\in[0,\infty)$,
\[
\sum_{i\in\mathcal{S}}\delta_{n^{-1}B_i^{n,\mu}}\rightarrow\sum_{x\in\mathcal{T}\backslash\{\emptyset\}}M_x\delta_{P_x}
\]
in distribution, with respect to the vague topology on the space of measures on $(0,1]$. That is, the measure applied to a finite collection of closed intervals in $(0,1]$ sees joint convergence. The random variables which appear in the limit are:
\begin{itemize}
\item
$(M_x)$ is a family of i.i.d. Poisson($\eta$) random variables;
\item
$(P_x)$ is from Lemma \ref{btst} and is independent of $(M_x)$.
\end{itemize}
\end{theorem}

\begin{remark}\label{msflfl}The mean site frequency spectrum, according to Theorem \ref{dslfm}'s limit, is
\ba
\mathbb{E}\left[\sum_{x\in\mathcal{T}\backslash\{\emptyset\}}M_x\delta_{P_x}(a,1)\right]=2\eta(a^{-1}-1),
\ea
which recovers the limit of Corollary \ref{lmfsfsm}.
\end{remark}
\begin{remark}\label{vsflfl}The variance of the site frequency spectrum, according to Theorem \ref{dslfm}'s limit, is bounded below by
\[
\text{Var}\left[\sum_{x\in\mathcal{T}\backslash\{\emptyset\}}M_x\delta_{P_x}(a,1)\right]\geq2\eta(a^{-1}-1).
\]
In particular, the coefficient of variation tends to infinity as $a\uparrow1$.
\end{remark}
The details of Remarks \ref{msflfl} and \ref{vsflfl} are given in Section \ref{pfrsfs}.

\begin{figure}
\begin{center}
\includegraphics{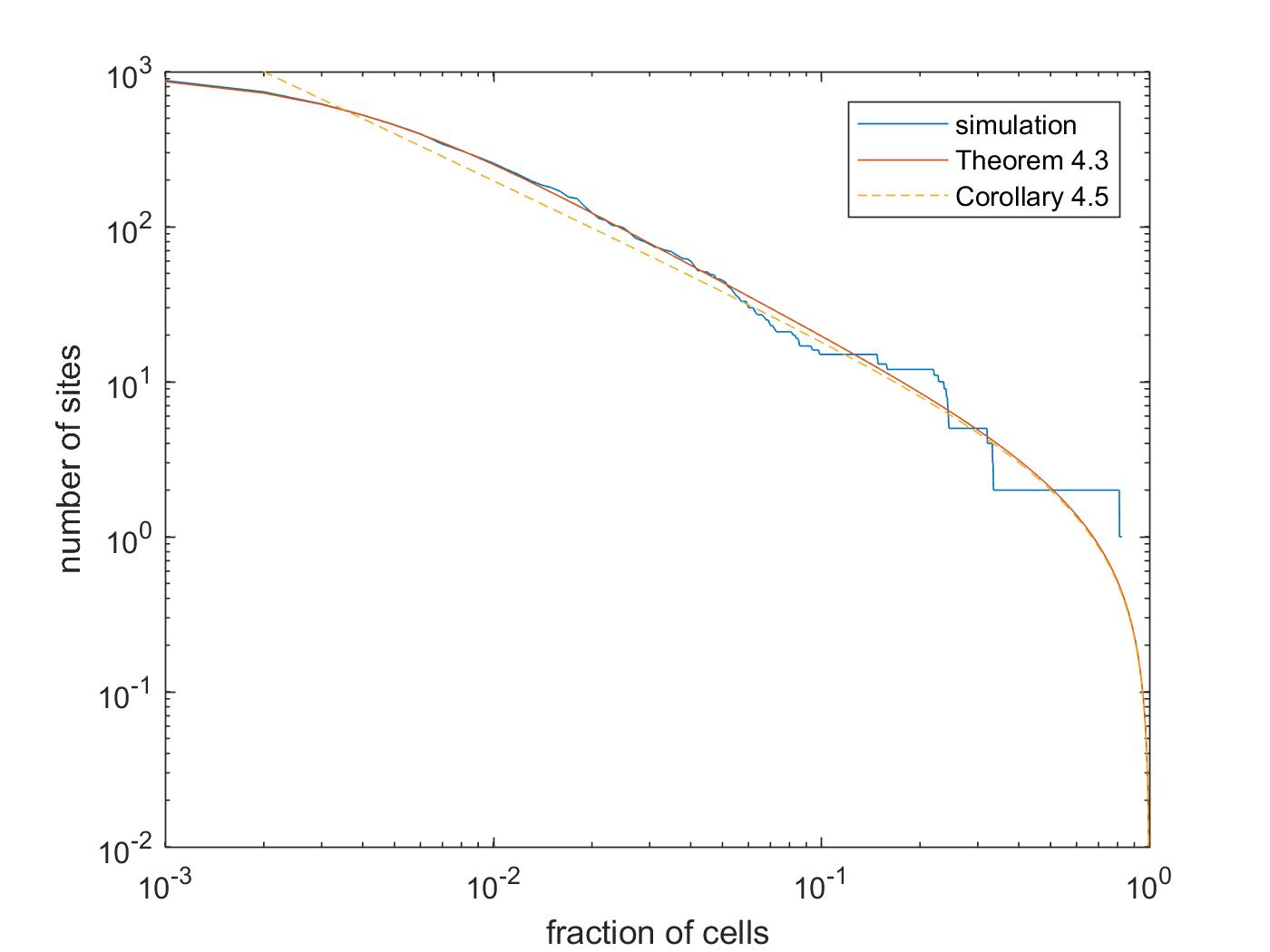}
	\caption{The site frequency spectrum is simulated a single time. The parameters are $\mu=10^{-3}$, $n=10^3$, and $|\mathcal{S}|=10^3$. The plot compares $\sum_{i\in\mathcal{S}}\delta_{n^{-1}B_i^{n,\mu}}(a,1)$ (simulation), $|\mathcal{S}|\mathbb{P}[n^{-1}B_1\in(a,1)]$ (Theorem \ref{llnldd}), and $2|\mathcal{S}|\mu(a^{-1}-1)$ (Corollary \ref{lmfsfsm}),  for $a\in(0,1)$.}
\end{center}
\end{figure}

\section{Generalisations}\label{gmrcs}
Motivated by biological reality, we introduce some generalisations: cell death, selection, and heterogeneous mutation rates.
\subsection{Model and notation}
Starting with one cell, the cell population grows according to a continuous-time multitype Markov branching process. The types are the genomes, elements of $\mathcal{G}=\mathcal{N}^\mathcal{S}$. It will be helpful to classify different types of genetic site. Partition the sites into \textit{neutral} and \textit{selective} sites:
\[
\mathcal{S}=\mathcal{S}_\text{neut}\cup\mathcal{S}_\text{sel},
\]
with $\mathcal{S}_\text{sel}\not=\emptyset$.
For a genome $v=(v_i)_{i\in\mathcal{S}}$, write $v'=(v_i)_{i\in\mathcal{S}_\text{sel}}$ for its restriction to the selective sites. Let $\alpha$ and $\beta$ be functions with domain $\mathcal{N}^{\mathcal{S}_\text{sel}}$ and range $[0,\infty)$. A cell with genome $v$ divides at rate $\alpha(v')$ (to be replaced by two daughter cells) and dies at rate $\beta(v')$. 

The initial cell is said to have genome $u$, which is assumed to give a positive growth rate: $\alpha(u')>\beta(u')$.

Consider a cell with genome $(v_i)_{i\in\mathcal{S}}$ dividing to give daughter cells with genomes $(V_i^{(1)})_{i\in\mathcal{S}}$ and $(V_i^{(2)})_{i\in\mathcal{S}}$. Conditional on $(v_i)$, the $V_i^{(r)}$ are independent over $i\in\mathcal{S}$ and $r\in\{1,2\}$, and
$$\mathbb{P}\left[V_i^{(r)}=\psi|(v_i)\right]=\mu_i^{v_i,\psi}.
$$

Slightly adapting previous notation, write
\[
\mu=\left(\mu_i^{\chi,\psi}\right)_{i\in\mathcal{S};\chi,\psi\in\mathcal{N}}
\]
for the collection of mutation rates. Now let's state the notation for mutation frequencies (for brevity, unlike in Section \ref{sec:mfn}, we shall do so in words). Write $B_i^{n,\mu}$ for the number of cells which are mutated at site $i$ when $n$ cells are first reached \textit{conditioned on the event that $n$ cells are reached}.
\subsection{Generalised Luria-Delbr{\"u}ck distribution}
Let $(\xi_k)_{k\in\mathbb{N}}$ be an i.i.d. sequence of exponentially distributed random variables with mean $\lambda^{-1}$. Let $(Y_k(\cdot))_{k\in\mathbb{N}}$ be an i.i.d. sequence, where $Y_1(\cdot)$ is a birth-death branching process with birth and death rates $a$ and $b$ respectively and initial condition $Y_1(0)=1$. Let $K$ be a Poisson random variable with mean $c$. The $(\xi_k)$, $(Y_k(\cdot))$, and $K$ are independent. The \textit{generalised Luria-Delbr{\"u}ck distribution} with parameters ($\lambda,a,b,c$) is defined as the distribution of
\[
B=\sum_{k=1}^KY_k(\xi_k).
\]
Its generating function
\[
\mathbb{E}z^B=\exp\left(c(b/a-1) F\left[1,\frac{\lambda}{a-b};1+\frac{\lambda}{a-b};\frac{b/a-z}{1-z}\right]\right)
\]
when $a>b$ is seen in \cite{ka,kl,ca}. Here $F$ is Gauss's hypergeometric function.

Taking parameters $(\lambda,\lambda,0,c)$ recovers the Luria-Delbr{\"u}ck distribution with parameter $c$.

The generalised Luria-Delbr{\"u}ck distribution with parameters ($\lambda,\lambda a,\lambda b,c$), for $\lambda>0$ and $a,b,c\geq0$, does not depend on $\lambda$. So one could define the distribution with $3$ rather than $4$ parameters. We choose $4$ for a cleaner interpretation of results.
\subsection{Results}
To begin, Theorem \ref{idlda} is generalised. The genomes whose only difference from the initial cell's genome is at site $i\in\mathcal{S}$,
\bq\label{osimg}
\mathcal{G}_{i}=\{v\in\mathcal{G}:\forall j\in\mathcal{S},(u_j\not=v_j\iff i=j)\},
\eq
will play a crucial role.
\begin{theorem}\label{ilddhmrs}
Take $n\rightarrow\infty$ and $n\mu_i^{\chi,\psi}\rightarrow\theta_i^{\chi,\psi}\in[0,\infty)$ for all $i\in\mathcal{S}$ and $\chi,\psi\in\mathcal{N}$ with $\chi\not=\psi$. Then
\[
\left(B_i^{n,\mu}\right)_{i\in\mathcal{S}}\rightarrow\left(\sum_{v\in\mathcal{G}_i}X_v\right)_{i\in\mathcal{S}}
\]
in distribution, where the $X_v$ are independent and have generalised Luria-Delbr{\"u}ck distributions with parameters
\[
\left(\alpha(u')-\beta(u'),\alpha(v'),\beta(v'),\frac{2\alpha(u')\theta_i^{u_i,v_i}}{\alpha(u')-\beta(u')}\right).
\]

\end{theorem}
In the next result, which generalises Theorem \ref{llnldd}, we keep the number of selective sites finite while taking the number of neutral sites to infinity. For this limit, mutation rates require consideration. Partition the set of neutral sites:
\[
\mathcal{S}_\text{neut}=\bigcup_{j\in J}\mathcal{S}(j),
\]
such that mutation rates and the initial genome's nucleotides are homogeneous on $\mathcal{S}(j)$ ($J$ is just some indexing set). Write $\mu^{\chi,\psi}(j)=\mu_i^{\chi,\psi}$ for the mutation rates of the sites $i\in\mathcal{S}(j)$. Write $u(j)=u_i$ for the initial genome's nucleotide at the sites $i\in\mathcal{S}(j)$.

\begin{theorem}\label{llnws}Take $n\rightarrow\infty$, $n\mu^{\chi,\psi}(j)\rightarrow\theta^{\chi,\psi}(j)\in[0,\infty)$, $|\mathcal{S}_\text{neut}|\rightarrow\infty$, and $|\mathcal{S}(j)|/|\mathcal{S}_\text{neut}|\rightarrow q(j)$, for all $j\in J$ and $\chi,\psi\in\mathcal{N}$ with $\chi\not=\psi$. Then
\ba
\frac{1}{|\mathcal{S}|}\sum_{i\in\mathcal{S}}\delta_{B_i^{n,\mu}}\overset{p}{\rightarrow} \sum_{j\in J}q(j)\Lambda(j)
\ea
where the $\Lambda(j)$ are generalised Luria-Delbr{\"u}ck distributions with parameters
\[
\left(\alpha(u')-\beta(u'),\alpha(u'),\beta(u'),\frac{2\alpha(u')}{\alpha(u')-\beta(u')}\sum_{\psi\in\mathcal{N}\backslash\{u(j)\}}\theta^{u(j),\psi}(j)\right).
\]
Convergence is on the space of probability measures on the non-negative integers equipped with the topology of weak convergence.
\end{theorem}
\subsection{Open problems}
To generalise Theorem \ref{dslfm} to a non-zero death rate, selection, and heterogeneous mutation rates, we conjecture the following.
\begin{conj}\label{lmfwsc}Take $n\rightarrow\infty$, $n\mu^{\chi,\psi}(j)\rightarrow\theta^{\chi,\psi}(j)\in[0,\infty)$, and $\mu^{\chi,\psi}(j)|\mathcal{S}(j)|\rightarrow\eta^{\chi,\psi}(j)\in[0,\infty)$, for all $j\in J$ and $\chi,\psi\in\mathcal{N}$ with $\chi\not=\psi$. Then
\[
\sum_{i\in\mathcal{S}}\delta_{n^{-1}B_i^{n,\mu}}\rightarrow\sum_{x\in\mathcal{T}}\sum_{r=1}^{R_x}M_{x,r}\delta_{P_x}
\]
in distribution, where convergence is in the same sense as Theorem \ref{dslfm}. The random variables which appear in the limit are:
\begin{itemize}
\item $(P_x)_{x\in\mathcal{T}\backslash\{\emptyset\}}$ is distributed as in Lemma \ref{btst} (and Theorem \ref{dslfm}), and $P_\emptyset=1$;
\item $(R_x)_{x\in\mathcal{T}\backslash\{\emptyset\}}$ is an i.i.d. family of geometric random variables with parameter $(\alpha(u')-\beta(u'))/(\alpha(u')+\beta(u'))$, and $R_\emptyset$ is independent of $(R_x)$ but with $R_\emptyset \overset{d}{=}R_0-1$;
\item $(P_x)$ is independent of $(R_x)$ if and only if $\beta(u')=0$;
\item
$(M_{x,r})_{x\in\mathcal{T},r\in\mathbb{N}}$ is an i.i.d. family of Poisson random variables with mean $\sum_j\sum_{\psi\not=u(j)}\eta^{u(j),\psi}(j)$, independent of $(P_x,R_x)$.
\end{itemize}
\end{conj}
See the \nameref{cdca} for a heuristic derivation of Conjecture \ref{lmfwsc}, which is based on a Yule spinal decomposition of the branching process.

Selection in cancer is a major research topic, and there have been attempts to infer selection from cancer genetic data \cite{drivers,matds,tavarerecent}. Pertinently, Theorem \ref{llnws} and Conjecture \ref{lmfwsc} suggest that selection may not be visible in mutation frequency data, which according to \cite{gs} is the case for around $1/3$ of tumours. However we have assumed that the number of selective sites is kept finite. According to \cite{drivers}, there are $3.4\times10^4$ selective sites at which mutations can positively affect growth rate. Thus insight could be gleaned, for example, by taking $|\mathcal{S}_\text{sel}|\rightarrow\infty$ with $\mu^\gamma|\mathcal{S}_\text{sel}|\rightarrow\eta$ for $\gamma\in(0,1]$.

\section{Mutations at finite numbers}\label{proofsecfinite}
In this section we prove results on mutations present in only a finite number of cells. In Subsections \ref{cgss} to \ref{wstarsigma} we prove Theorems \ref{idlda} and \ref{ilddhmrs} (where $\mathcal{S}$ is finite). In Subsection \ref{llnpss} we prove Theorems \ref{llnldd} and \ref{llnws} (where $|\mathcal{S}|$ tends to infinity).
\subsection{Counting genomes}\label{cgss}
 Assuming mutation rates $\mu=(\mu_i^{\chi,\psi})_{i\in\mathcal{S};\chi\psi\in\mathcal{N}}$, write
\bq\label{tbbn}
X^\mu_v(t)
\eq
for the number of cells with genome $v\in\mathcal{G}$ at time $t\geq0$. (Recall that the initial condition is $X_v^\mu(0)=\delta_{u,v}$.) Write
\[
\sigma_n^\mu=\min\left\{t\geq0:\sum_{v\in\mathcal{G}}X^\mu_v(t)=n\right\}
\]
for the time at which $n\in\mathbb{N}$ cells are reached, and use the convention $\min\emptyset=\infty$.

Recall from (\ref{osimg}) that $\mathcal{G}_{i}$ is the subset of genomes with exactly one mutation which is at site $i$. Write
\[
\mathcal{G}_{\geq2}=\left\{v\in\mathcal{G}:|\{i\in\mathcal{S}:v_i\not=u_i\}|\geq2\right\}
\]
for the subset of genomes with at least two mutations.
\begin{theorem}\label{mt}
Take $n\rightarrow\infty$ and $n\mu_i^{\chi,\psi}\rightarrow\theta_i^{\chi,\psi}\in[0,\infty)$ for $i\in\mathcal{S}$ and $\chi,\psi\in\mathcal{N}$ with $\chi\not=\psi$. Then
\[
\left[(X_v^\mu(\sigma^\mu_n))_{v\in\mathcal{G}\backslash\{u\}}|\sigma^\mu_n<\infty\right]\rightarrow(X_v)_{v\in\mathcal{G}\backslash\{u\}}
\]
in distribution, where the $X_v$ are independent and distributed according to:
\begin{itemize}
\item
for $v\in\mathcal{G}_{i}$, $X_v$ has generalised Luria-Delbr{\"u}ck distribution with parameters
\[
\left(\alpha(u')-\beta(u'),\alpha(v'),\beta(v'),\frac{2\alpha(u')\theta^{u_i,v_i}_i}{\alpha(u')-\beta(u')}\right);
\]
\item
for $v\in\mathcal{G}_{\geq2}$, $X_v=0$.
\end{itemize}
\end{theorem}
Theorem \ref{mt} says that cells with at least two mutated sites are non-existent. However simulations and biology tell the opposite story, that cells typically have many mutated sites. This apparent contradiction comes because, while the population size and mutation rate reciprocal converge to infinity, the number of sites is kept finite. So the result only makes sense if one is considering a small subset of the billions of sites.

The mutation frequencies are
\[
\left(B_i^{n,\mu}\right)_{i\in\mathcal{S}}=\left(\sum_{\substack{v\in\mathcal{G}\\v_i\not= u_i}}X_v^{\mu}(\sigma_n^{\mu})\right)_{i\in\mathcal{S}}
\]
conditional on the event $\{\sigma_n^\mu<\infty\}$. Therefore Theorem \ref{mt}, via the continuous mapping theorem, implies Theorems \ref{idlda} and \ref{ilddhmrs}.

Theorem \ref{mt}'s proof is rather lengthy. So, before jumping in with the technical details, let's give an overview.

In Subsection \ref{constru} we present a construction of $(X_v^{\mu}(\sigma^{\mu}_n))_{v\in\mathcal{G}}$. The construction will ultimately illuminate the importance of various subpopulations and the mutations between them. Of particular importance is the \textit{primary} subpopulation, which is defined as those unmutated cells which have an unbroken lineage of unmutated cells going back to the initial cell. The primary subpopulation, in the limit, grows deterministically and exponentially.

In Subsections \ref{dmecbn} and \ref{asco} we show that several events are negligible: primary cells divide to give two mutant daughters; primary cells mutate at multiple sites at once; mutated cells receive further mutations, including backwards mutations. With these events neglected, the situation is pleasingly simplified. The primary subpopulation seeds, as a Poisson process with exponential intensity, single-site mutant subpopulations. The mutant subpopulations grow without further mutations, independently. This gives independent Luria-Delbr{\"u}ck distributions. Finally, in Subsection \ref{wstarsigma} we condition on the event that the population reaches $n$ cells.

Although the proof's overview may sound simple, the details are less so. The reader will find that the random time $\sigma_n^{\mu}$ shoulders a large responsibility for complexity.

\subsection{Construction}\label{constru}

Additional notation to be used in the proof: for $v\in\mathcal{G}$,
\[
e_v=(\delta_{v,w})_{w\in\mathcal{G}}
\]
is the element of $(\mathbb{N}_0)^\mathcal{G}$ denoting that there is one genome $v$ and zero other genomes. 

Let 
\[
\left[\mu_n\right]_{n\in\mathbb{N}}=\left[\left(\mu_{n,i}^{\chi,\psi}\right)_{i\in\mathcal{S};\chi,\psi\in\mathcal{N}}\right]_{n\in\mathbb{N}}
\] be a sequence of mutation rates. Assume that
\[
\lim_{n\rightarrow\infty}n\mu^{\chi,\psi}_{n,i}=\theta_i^{\chi,\psi}\in[0,\infty)
\]
for $\chi\not=\psi$.

Fix $n\in\mathbb{N}$. For $v,w\in\mathcal{G}$, write
\bq\label{probdcg}
p_n(v,w)=\prod_{i\in\mathcal{S}}\mu_{n,i}^{u_i,v_i}\mu_{n,i}^{u_i,w_i}
\eq
for the probability that a cell with genome $u$ which divides, gives daughters with genomes $v,w$ (which implies that we have assumed an ordering of the daughters - the first has genome $v$ and the second has genome $w$). Now the construction of $(X_v^{\mu_n}(\sigma^{\mu_n}_n))_{v\in\mathcal{G}}$ begins. For the foundational step, introduce the following random variables on a fresh probability space.
\begin{enumerate}
\item
\[(Z^n(t))_{t\geq0}\]
is a birth-death branching process with birth and death rates
\[
\alpha_n:=\alpha(u')p_n(u,u)
\]
and 
\[
\beta_n:=\beta(u')+\alpha(u')\sum_{v,w\in\mathcal{G}\backslash\{u\}}p_n(v,w).
\]
The initial condition $Z^n(0)=1$ is assumed.
\item For $j\in\mathbb{N}$,
\[E_j^n
\] 
are $\{\emptyset\}\cup(\mathcal{G}\backslash\{u\})^2$-valued random variables, with
\[
\mathbb{P}[E^n_j=\emptyset]=\frac{\beta(u')}{\beta_n},
\]
and for $v,w\in\mathcal{G}\backslash\{u\}$
\[
\mathbb{P}[E^n_j=(v,w)]=\frac{\alpha(u')p_n(v,w)}{\beta_n}.
\]
\item
For $v\in\mathcal{G}\backslash\{u\}$ and $j\in\mathbb{N}$,
\[
\mathcal{Y}^n_{v,j}(\cdot)\]
is a $(\mathbb{N}_0)^\mathcal{G}$-valued Markov process, with the same transition rates as $(X^{\mu_n}_x(\cdot))_{x\in\mathcal{G}}$ (defined in (\ref{tbbn})) and with the initial condition $\mathcal{Y}^n_{v,j}(0)=e_v$.
\item For $v,w\in\mathcal{G}\backslash\{u\}$ and $j\in\mathbb{N}$,
\[\mathcal{Y}^n_{v,w,j}(\cdot)\]
is a $(\mathbb{N}_0)^\mathcal{G}$-valued Markov process, with the same transition rates as $(X^{\mu_n}_x(\cdot))_{x\in\mathcal{G}}$ and with the initial condition $\mathcal{Y}^n_{v,w,j}(0)=e_v+e_w$.
\item For $v\in\mathcal{G}$,
\[(N_v(t))_{t\geq0}\]
are Poisson counting processes with rate $1$.

\end{enumerate}
The random variables
\bq
\left[Z^n(\cdot),E^n_j,\mathcal{Y}^n_{v,j}(\cdot),\mathcal{Y}^n_{v,w,j}(\cdot),N_v(\cdot)\right]\label{randomv}
\eq are assumed to be independent ranging over $v,w,j$.

Let's explain the meaning of the random variables introduced so far. $Z^n(\cdot)$ represents the `primary' subpopulation - which we define as the type $u$ cells whose ancestors are all of type $u$. That is to say, there is an unbroken lineage of type $u$ cells between any primary cell and the initial cell. The rate, $\alpha_n$, that a primary cell gives birth to another primary cell, is simply the type $u$ division rate multiplied by the probability that no mutation occurs in either daughter cell. The rate, $\beta_n$, that a primary cell is removed, is the rate that a type $u$ cell dies plus the rate that a type $u$ cell divides to produce two mutant daughter cells.

The $E_j^n$ describe what happens at the $j$th downstep in the primary subpopulation trajectory. If $E_j^n=\emptyset$, then the downstep is a primary cell death. If $E_j^n=(v,w)$, then the downstep is a primary cell dividing to produce two mutant daughter cells of types $v$ and $w$.

Sometimes a primary cell divides to produce one primary cell and one mutant cell of type $v$. For the $j$th time that this occurs, $\mathcal{Y}^n_{v,j}(t)$ is the vector which counts the cells with each genome amongst the descendants of that type $v$ cell, time $t$ after its birth.

Sometimes a primary cell divides to produce two mutant cells of types $v$ and $w$. For the $j$th time that this occurs, $\mathcal{Y}^n_{v,w,j}(t)$ is the vector which counts the cells with each genome amongst the descendants of the two mutants time $t$ after their birth.

The $N_v(\cdot)$ will soon be rescaled in time to represent the times at which primary cells divide to produce one primary cell and one cell with genome $v$.

The random variables introduced so far, seen together in (\ref{randomv}), provide all the necessary ingredients for the construction of $(X_v^{\mu_n}(\sigma_n^{\mu_n}))_{v\in\mathcal{G}}$. Now we build upon these founding objects, defining further random variables.

\begin{enumerate}[resume]
\item
For $v\in\mathcal{G}\backslash\{u\}$ and $t\geq0$,
\bq\label{knzd}
K^n_v(t)=N_v\left(2p_n(u,v)\alpha(u')\int_0^tZ^n(s)ds\right).
\eq
\item
For $j\in\mathbb{N}$ and $v\in\mathcal{G}\backslash\{u\}$,
\[
T^n_{v,j}=\min\{t\geq0:K^n_v(t)=j\}.
\]
\item
\[
S^n_1=\min\{t\geq0:Z^n(t)-Z^n(t^-)=-1\},
\]
and then for $j>1$, recursively,
\[
S^n_j=\min\{t>S^n_{j-1}:Z^n(t)-Z^n(t^-)=-1\}.
\]
(Here $Z^n(t^-):=lim_{s\uparrow t}Z^n(s)$.)

\item For $v,w\in\mathcal{G}\backslash\{u\}$,
\[
T^n_{v,w,1}=\min\{S^n_j:j\in\mathbb{N},E^n_j=(v,w)\},
\]
and then for $j>1$, recursively,
\[
T^n_{v,w,j}=\min\{S^n_j:j\in\mathbb{N},S^n_j>T^n_{v,w,j-1},E^n_j=(v,w)\}.
\]
\item For $v,w\in\mathcal{G}\backslash\{u\}$, and $t\geq0$,
\ba
K^n_{v,w}(t)=|\{j\in\mathbb{N}:T^n_{v,w,j}\leq t\}|.
\ea
\end{enumerate}
Let's explain the meaning of the new random variables. The $K^n_v(t)$ specify the number of times before time $t$ that primary cells have divided to produce one primary cell and one type $v$ cell. Let's check that this interpretation makes sense. Conditioned on the trajectory of $Z^n(\cdot)$, $K_v^n(\cdot)$ is certainly a Markov process, and increases by $1$ at rate $2p_n(u,v)\alpha(u')Z^n(t)$ - i.e. the rate at which primary cells divide multiplied by the probability that exactly one daughter cell is primary and one is type $v$.

$S^n_j$ is the time of the $j$th downstep of the primary subpopulation size. Then $T^n_{v,w,j}$ is the time of the $j$th primary cell division which produces two mutant cells of types $v$ and $w$. Note that a primary cell division which produces two mutant cells neccessarily coincides with a downstep in the primary subpopulation size. $K^n_{v,w}(t)$ is the number of primary cell divisions before time $t$ which produce cells of types $v$ and $w$.

The reader might question why we have decided to construct the `single mutation' times and the `double mutation' times so differently. The reason for the difference is that single and double mutations will play different roles in the limit, and require different techniques for the proof.

At last the construction reaches its d\'enouement.
\begin{enumerate}[resume]
\item For $t\geq0$,
\ba
\mathcal{X}^n(t)&=&Z^n(t)e_u\\
&&+\sum_{v\in\mathcal{G}\backslash\{u\}}\sum_{j=1}^{K^n_v(t)}\mathcal{Y}^n_{v,j}(t-T^n_{v,j})\label{repr}\\
&&+\sum_{v,w\in\mathcal{G}\backslash\{u\}}\sum_{j=1}^{K^n_{v,w}(t)}\mathcal{Y}^n_{v,w,j}(t-T^n_{v,w,j}).
\ea
\item
\[
\sigma_n=\min\{t\geq0:|\mathcal{X}^n(t)|=n\},
\]
where $|\cdot|$ is the $l_1$-norm on $\mathbb{R}^\mathcal{G}$.
\end{enumerate}
Note that $\mathcal{X}^n(\cdot)$ has the same distribution as $(X_v^{\mu_n}(\cdot))_{v\in\mathcal{G}}$; both objects are Markov processes on $(\mathbb{N}_0)^{\mathcal{G}}$, whose initial conditions and transition rates coincide. 

Next we will show that certain elements of the construction converge in distribution. Convergence will sometimes be in the Skorokhod sense. For notation, write $\mathbb{D}(I,R)$ for the space of c\`adl\`ag functions from an interval $I\subset[0,\infty)$ to a metric space $R$ (which will always be complete and separable). The space $\mathbb{D}(I,R)$ is equipped with the standard Skorokhod topology. Less standard, we will also consider the space $\mathbb{D}([0,\infty],R)$, which is defined by identification with $\mathbb{D}([0,1],R)$. Let's be specific. Define $\omega:[0,1]\rightarrow[0,\infty]$ by
$$
\omega:s\mapsto\begin{cases}-\log(1-s),\quad&s\in[0,1);\\
\infty,\quad&s=1.\end{cases}
$$
For $f\in\mathbb{D}([0,1],R)$ define $\hat{w}(f)=f\circ\omega^{-1}$. Then the space $\mathbb{D}([0,\infty],R)$ is the image of $\hat{w}$ equipped with the induced topology. Note that according to this definition,  for all $z\in\mathbb{D}([0,\infty],R)$, $\lim_{t\rightarrow\infty}z(t)$ exists. In fact we will only consider $z\in\mathbb{D}([0,\infty],R)$ with $\lim_{t\rightarrow\infty}z(t)=z(\infty)$.
\begin{lemma}\label{sczti}
As $n\rightarrow\infty$,
\[
(e^{-\lambda_n t}Z^n(t))_{t\in[0,\infty]}\rightarrow(e^{-\lambda t}Z^*(t))_{t\in[0,\infty]}
\]
in distribution, on the space $\mathbb{D}([0,\infty],\mathbb{R})$. Here
\[
\lambda_n=\alpha_n-\beta_n
\]
is the growth rate of the primary cell population;
\[
\lambda=\alpha(u')-\beta(u')
\]
is the large $n$ limit of $\lambda_n$; and $Z^*(\cdot)$ is a birth-death branching process with birth and death rates $\alpha(u')$ and $\beta(u')$.
\end{lemma}
\begin{remark}\label{marlim}
The processes of Lemma \ref{sczti} are defined at $t=\infty$. For $n$ large enough that $\lambda_n>0$,
\ba
e^{-\lambda_n\infty}Z^n(\infty):=\lim_{t\rightarrow\infty}e^{-\lambda_n t}Z^n(t)=W^n,
\ea
and
\ba
e^{-\lambda\infty}Z^*(\infty):=\lim_{t\rightarrow\infty}e^{-\lambda t}Z^*(t)=W^*.
\ea
The limits $W^n$ and $W^*$ exist and are finite almost surely, which is a classic branching process result \cite{an}.
\end{remark}
\begin{proof}[Proof of Lemma \ref{sczti}]
The transition probabilities of the $\left(e^{-\lambda_n t}Z^n( t)\right)_{t\in[0,\infty]}$ and $\left(e^{-\lambda  t}Z^*( t)\right)_{t\in[0,\infty]}$ are well-known \cite{an,rdbp,ca}. These transition probabilities depend continuously on the birth and death rates, so finite-dimensional convergence is given. To show tightness we shall use Aldous's criterion \cite{asc}. Extend $ \omega(s)$ to $s\in[0,2]$ by setting $ \omega(s)= \omega(1)=\infty$ for $s\in[1,2]$. Write
\ba
M^n(s)&=&e^{-\lambda_n \omega(s)}Z^n( \omega(s))
\ea
for $s\in[0,2]$. Let $(\rho_n)$ be a sequence of $[0,1]$-valued stopping times with respect to $(M^n(\cdot))$. Let $(\delta_n)$ be a positive deterministic sequence converging to zero. Then, writing $\mathcal{F}_{\rho_n}$ for the sigma-algebra generated by $M^n(\cdot)$ up to time $\rho_n$,
\ba
\mathbb{E}[(M^n(\rho_n+\delta_n)-M^n(\rho_n))^2|\mathcal{F}_{\rho_n}]&=&\mathbb{E}[M^n(\rho_n+\delta_n)^2|\mathcal{F}_{\rho_n}]-M^n(\rho_n)^2\\
&=&\left(e^{-\lambda_n\omega(\rho_n)}-e^{-\lambda_n\omega(\rho_n+\delta_n)}\right)\\
&&\quad\quad\times\frac{\alpha_n+\beta_n}{\lambda_n}M^n(\rho_n),
\ea
where the last equality comes thanks to the fact that
\[
M^n(s)^2+\frac{\alpha_n+\beta_n}{\lambda_n}e^{-\lambda_n\omega(s)}M^n(s)
\]
is a martingale. But
\ba
e^{-\lambda_n\omega(\rho_n)}-e^{-\lambda_n\omega(\rho_n+\delta_n)}&=&(1-\rho_n)^{\lambda_n}-1_{\{1-\rho_n-\delta_n\geq0\}}(1-\rho_n-\delta_n)^{\lambda_n}\\
&\leq&\max\{\lambda_n\delta_n,\delta_n^{\lambda_n}\}.
\ea
Now,
\ba
\mathbb{E}[(M^n(\rho_n+\delta_n)-M^n(\rho_n))^2]&\leq&\max\{\lambda_n\delta_n,\delta_n^{\lambda_n}\}\frac{\alpha_n+\beta_n}{\lambda_n}\mathbb{E}M^n(\rho_n)\\
&=&\max\{\lambda_n\delta_n,\delta_n^{\lambda_n}\}\frac{\alpha_n+\beta_n}{\lambda_n}.\ea
Take $n\rightarrow\infty$ to see that $M^n(\rho_n+\delta_n)-M^n(\rho_n)$ converges to zero in $L_2$ and hence in probability, thus satisfying Aldous's criterion.
\end{proof}
\begin{lemma}\label{easys}
As $n\rightarrow\infty$,
\[
\mathcal{Y}^n_{v,j}(\cdot)\rightarrow Y_{v,j}(\cdot)e_v
\]
in distribution, where $Y_{v,j}(\cdot)$ is a birth-death branching process with birth and death rates $\alpha(v')$ and $\beta(v')$ and initial condition $Y_{v,j}(0)=1$. Convergence is on the space $\mathbb{D}([0,\infty),\mathbb{R}^\mathcal{G})$.
\begin{proof}
It is enough to note that the transition rates converge (see for example page 262 of \cite{ek}).
\end{proof}
\end{lemma}
\begin{lemma}\label{skcodm}As $n\rightarrow\infty$,
\[
\left(\sum_{j\leq kn^{3/2}}1_{\{E_j^n\not=\emptyset\}}\right)_{k\in\mathbb{N}}\rightarrow(0)_{k\in\mathbb{N}}
\]
in distribution, on the space $\mathbb{R}^\mathbb{N}$.
\begin{proof}
Note that
\[
\lim_{n\rightarrow\infty}n^{3/2}\mathbb{P}[E^n_j\not=\emptyset]=0.
\]
Then
\ba
\mathbb{P}\left[\sum_{j\leq kn^{3/2}}1_{\{E_j^n\not=\emptyset\}}=0;k=1,..,r\right]&=&\mathbb{P}\left[E_j^n=\emptyset;j\leq rn^{3/2}\right]\\
&=&\left(1-\mathbb{P}\left[E_j^n\not=\emptyset\right]\right)^{\lfloor rn^{3/2}\rfloor}\\
&\rightarrow&1.
\ea
\end{proof}
\end{lemma}
\begin{remark}The number $3/2$ which appears in Lemma \ref{skcodm} is not special. It only matters that $3/2\in(1,2)$. The relevance of the result will be seen in Section \ref{dmecbn}.
\end{remark}
\begin{lemma}\label{jointconvergencefors}
As $n\rightarrow\infty$,
$$
\left[(e^{-\lambda_nt}Z^n(t))_{t\in[0,\infty]},\left((\mathcal{Y}^n_{v,j}(t))_{t\in[0,\infty)}\right)_{v\in\mathcal{G}\backslash\{u\},j\in\mathbb{N}},\left(\sum_{j\leq kn^{3/2}}1_{\{E_j^n\not=\emptyset\}}\right)_{k\in\mathbb{N}}\right]
$$
converges in distribution to
$$
\left[(e^{-\lambda t}Z^*(t))_{t\in[0,\infty]};\left((Y_{v,j}(t)e_v)_{t\in[0,\infty)}\right)_{v\in\mathcal{G}\backslash\{u\},j\in\mathbb{N}},\left(0\right)_{k\in\mathbb{N}}\right]
$$
on
$$\mathbb{D}([0,\infty],\mathbb{R})\times\mathbb{D}\left([0,\infty),\mathbb{R}^\mathcal{G}\right)^{\mathcal{G}\backslash\{u\}\times\mathbb{N}}\times\mathbb{R}^\mathbb{N},
$$
where the $Z^*(\cdot)$ and $Y_{v,j}(\cdot)$ are independent.
\begin{proof}
The convergence seen in Lemmas \ref{sczti}, \ref{easys}, and \ref{skcodm} is joint convergence over the product space due to independence.
\end{proof}
\end{lemma}
We are yet to say how the random variables in (\ref{randomv}) are jointly distributed over $n\in\mathbb{N}$. In fact, the choice of this joint distribution over $n\in\mathbb{N}$ has no relevance to the statement of Theorem \ref{mt}. Hence the choice can be freely made, in a way that streamlines the proof. We assume that:
\bq\label{ascpc}
\lim_{n\rightarrow\infty}(e^{-\lambda_nt}Z^n(t))_{t\in[0,\infty]}=(e^{-\lambda t}Z^*(t))_{t\in[0,\infty]}
\eq
almost surely, on the space $\mathbb{D}([0,\infty],\mathbb{R})$;
\bq\label{ascmc}
\lim_{n\rightarrow\infty}(\mathcal{Y}^n_{v,j}(t))_{t\in[0,\infty)}=(Y_{v,j}(t)e_v)_{t\in[0,\infty)}
\eq almost surely, on the space $\mathbb{D}([0,\infty),\mathbb{R}^\mathcal{G})$, for $v\in\mathcal{G}\backslash\{u\}$ and $j\in\mathbb{N}$; and
\bq\label{asskcdm}
\left(\sum_{j\leq kn^{3/2}}1_{\{E_j^n\not=\emptyset\}}\right)_{k\in\mathbb{N}}\rightarrow(0)_{k\in\mathbb{N}}
\eq
almost surely, on the space $\mathbb{R}^\mathbb{N}$.

To justify that it is possible to have constructed the random variables in such a way that (\ref{ascpc}), (\ref{ascmc}), and (\ref{asskcdm}) hold, one can bring in Skorokhod's Representation Theorem, to use with Lemma \ref{jointconvergencefors}.

\subsection{Neglecting double mutations}\label{dmecbn}
Call the event that a primary cell divides to produce two mutant cells a `double mutation'. Recall that double mutations are represented by the events $\{E_j^n=(v,w)\}$, which occur at the times $S_j^n$ when the primary cell population steps down in size. In order to comment on double mutations, we will first prove a rather crude upper bound for the number of downsteps in the primary cell population trajectory. Write
\bq\label{taundef}
\tau_n:=\min\{t\geq0:Z^n(t)\in\{0,n\}\},
\eq
for the time at which the primary cell population hits $0$ or $n$. Write
\[
D_n:=\left|\{j\in\mathbb{N}:S_j^n\leq\tau_n\}\right|
\]
for the number of downsteps in the primary cell population before time $\tau_n$.
\begin{lemma}\label{cubnds}
\[\sup_{n\in\mathbb{N}}n^{-3/2}D_n<\infty
\]
almost surely.
\begin{proof}
For each $n\in\mathbb{N}$, let $(R_j^n)_{j\in\mathbb{N}}$ be a sequence of i.i.d. random variables with
\ba
\mathbb{P}[R^n_j=x]=\begin{cases} \alpha_n/(\alpha_n+\beta_n),\quad x=1;\\ \beta_n/(\alpha_n+\beta_n),\quad x=-1;\end{cases}
\ea
so
\[
\left(1+\sum_{j=1}^kR_j^n\right)_{k\in\mathbb{N}}
\]
is a random walk, whose distribution matches that of the discrete-time embedded chain of $Z^n(\cdot)$. Write
\[
\rho_n=\min\left\{k\in\mathbb{N}:1+\sum_{j=1}^kR_j^n\in\{0,n\}\right\}
\]
for the number of steps until the walk hits $n$ or $0$. Then the number of downsteps before hitting $n$ or $0$ is
\[
D_n\overset{d}{=}\sum_{j=1}^{\rho_n}1_{\{R^n_j=-1\}}\leq\rho_n.
\]
Therefore we can bound the tail of $D_n$'s distribution:
\[
\mathbb{P}[D_n> n^{3/2}]\leq \mathbb{P}[\rho_n> n^{3/2}].
\]
But $\{\rho_n> n^{3/2}\}\subset\{1+\sum_{j=1}^{\lfloor n^{3/2}\rfloor}R_j^n<n\}$, so
\bq
\mathbb{P}[D_n> n^{3/2}]&\leq&\mathbb{P}\left[1+\sum_{j=1}^{\lfloor n^{3/2}\rfloor}R_j^n<n\right]\nonumber\\
&\leq&\mathbb{P}\Bigg[\left(\sum_{j=1}^{\lfloor n^{3/2}\rfloor}R_j^n- \frac{\lfloor n^{3/2}\rfloor\lambda_n}{\alpha_n+\beta_n}\right)^2\\
&&\quad\quad\quad>\left(\frac{\lfloor n^{3/2}\rfloor\lambda_n}{\alpha_n+\beta_n}+1-n\right)^2\Bigg]\label{wdtims}\nonumber\\
&\leq&\left(\frac{\lfloor n^{3/2}\rfloor\lambda_n}{\alpha_n+\beta_n}+1-n\right)^{-2}\text{Var}\left[\sum_{j=1}^{\lfloor n^{3/2}\rfloor}R_j^n\right]\label{chshi}\\
&\leq&cn^{-3/2}\label{bwn32},
\eq
for some constant $c>0$. Inequality (\ref{wdtims}) holds for large enough $n$ and Inequality (\ref{chshi}) is Chebyshev's inequality. Finally, (\ref{bwn32}) gives that
\[
\sum_{n\in\mathbb{N}}\mathbb{P}[D_n> n^{3/2}]<\infty,
\]
and the result is proven by Borel-Cantelli.
\end{proof}
\end{lemma}
Now it is to be seen that double mutations occurring before time $\tau_n$ can be neglected.
\begin{lemma}\label{kyztz}Let $v,w\in\mathcal{G}\backslash\{u\}$. As $n\rightarrow\infty$,
\[
K^n_{v,w}(\tau_n)\rightarrow0
\]
almost surely.

\begin{proof}From Lemma \ref{cubnds}, $C:=\sup_{n\in\mathbb{N}}n^{-3/2}D_n<\infty$. Then
\ba
K^n_{v,w}(\tau_n)&=&\sum_{j=1}^{D_n}1_{\{E_j^n=(v,w)\}}\\
&\leq&\sum_{j=1}^{\lceil Cn^{3/2}\rceil}1_{\{E_j^n=(v,w)\}}\\
&\leq&\sum_{j=1}^{\lceil Cn^{3/2}\rceil}1_{\{E_j^n\not=\emptyset\}}.
\ea
By (\ref{asskcdm}) this converges to zero as $n\rightarrow\infty$.
\end{proof}
\end{lemma}

\subsection{Convergence of genome counts}\label{asco}
The purpose of this section is to show that $\mathcal{X}^n(\sigma_n)$ converges when conditioned on the event $\{W^*>0\}$ ($W^*$ is defined in Remark \ref{marlim}). The times $\tau_n$ (defined in (\ref{taundef})) will play the role of a helpful stepping stone in the proof. 
\begin{lemma}\label{cwd}
Condition on $\{W^*>0\}$. Then, almost surely,
\begin{enumerate}
\item there exists $n_0$ such that for all $n\geq n_0$, $Z^n(\tau_n)=n$; and
\item $\lim_{n\rightarrow\infty}\tau_n=\infty$.
\end{enumerate}

\begin{proof}To see the first statement, observe that there exists $n_0$ such that for all $n\geq n_0$, $W^n>W^*/2>0$. For such $n$, $\lim_{t\rightarrow\infty} Z^n(t)=\infty$, and hence $Z^n(\cdot)>0$. To see the second statement, suppose for a contradiction that there exists a bounded subsequence $(\tau_{n_k})\subset[0,C]$. Then, for large enough $k$,
\[
n_k=Z^{n_k}(\tau_{n_k})\leq\sup_{n\in\mathbb{N}}\sup_{t\in[0,C]}Z^n(t).
\]
The left hand side of the inequality is unbounded over $k$. On the other hand, the right hand side, which does not depend on $k$, is finite thanks to (\ref{ascpc}).
\end{proof}
\end{lemma}
\begin{lemma}\label{consmt}Condition on $\{W^*>0\}$. Suppose that $(a_n)_{n\in\mathbb{N}}$ is a sequence of real-valued random variables on the same probabiity space as everything else, with
$$
\lim_{n\rightarrow\infty}a_n=\infty,
$$
\[
a_n\leq\tau_n
\]
for each $n$, and
\[
\lim_{n\rightarrow\infty}(a_n-\tau_n)=l\in[-\infty,0]
\]
almost surely. Then, almost surely,
\ba
\lim_{n\rightarrow\infty}K_v^n(a_n)= \begin{cases}K_v^*(l),\quad &v\in\mathcal{G}_{i},i\in\mathcal{S};\\
0,\quad&v\in\mathcal{G}_{\geq2};
\end{cases}
\ea
where
\[
K_v^*(s)=N_v(2\lambda^{-1}\alpha(u')\theta_i^{u_i,v_i} e^{\lambda s}).
\]
Moreover, for $v\in\mathcal{G}_i$ and $j\in\mathbb{N}$,
$$
\lim_{n\rightarrow\infty}\left(a_n-T^n_{v,j}\right)= l-T^*_{v,j}
$$
almost surely,
where
$$
T^*_{v,j}=\min\{s\in\mathbb{R}:K_v^*(s)=j\}.
$$
\begin{proof}
Let $t\in\mathbb{R}$. Since $Z^n(\tau_n)=n$,
$$
n^{-1}\int_0^{a_n+t}Z^n(s)ds=\int_{-a_n}^t\frac{Z^n(a_n+s)}{e^{\lambda_n(a_n+s)}}\frac{e^{\lambda_n\tau_n}}{Z^n(\tau_n)}e^{\lambda_n(a_n-\tau_n+s)}ds.
$$
Thanks to (\ref{ascpc}):
\begin{enumerate}
\item for any sequence $(t_n)$ which converges to infinity, $\lim_{n\rightarrow\infty}e^{-\lambda_nt_n}Z^n(t_n)=W^*$ almost surely; and
\item
$\sup_{n\in\mathbb{N}}\sup_{t\in[0,\infty]}e^{-\lambda_nt}Z^n(t)<\infty$ almost surely.
\end{enumerate}
So, using dominated convergence,
\ba
\lim_{n\rightarrow\infty}n^{-1}\int_0^{a_n+t}Z^n(s)ds=\lambda^{-1}e^{\lambda(l+t)}
\ea
almost surely. Note also that
\ba
\lim_{n\rightarrow\infty}np_n(u,v)=\begin{cases}\theta_i^{u_i,v_i},\quad &v\in\mathcal{G}_{i};\\0,\quad &v\in\mathcal{G}_{\geq2}.
\end{cases}
\ea
Then
$$
\lim_{n\rightarrow\infty}2p_n(u,v)\alpha(u')\int_0^{a_n+t}Z^n(s)ds=2\lambda^{-1}\alpha(u')\theta_i^{u_i,v_i}e^{\lambda(l+t)}.
$$
Hence, recalling (\ref{knzd}),
$$
K^n_v(a_n+t)=N_v\left(2p_n(u,v)\alpha(u')\int_0^{a_n+t}Z^n(s)ds\right)
$$
converges almost surely to
$$
K^*_v(l+t)=N_v\left(2\lambda^{-1}\alpha(u')\theta_i^{u_i,v_i}e^{\lambda(l+t)}\right),
$$ because $N_v(\cdot)$ is almost surely continuous at any fixed point.

Finally we check convergence of the $a_n-T^n_{v,j}$. Let $\epsilon>0$. For sufficiently large $n$,
$$
K^n_{v}(a_n-l+T^*_{v,j}+\epsilon)=K^*_{v}(T^*_{v,j}+\epsilon)\geq j;
$$
so
$$a_n-l+T_{v,j}^*+\epsilon\geq T_{v,j}^n,$$
or equivalently $$a_n-T_{v,j}^n\geq l-T^*_{v,j}-\epsilon.$$
The argument is now repeated for an upper bound. For sufficiently large $n$,
$$
K^n_{v}(a_n-l+T^*_{v,j}-\epsilon)=K^*_{v}(T^*_{v,j}-\epsilon)< j;
$$
so
$$a_n-l+T_{v,j}^*-\epsilon<T_{v,j}^n,$$
or equivalently $$a_n-T_{v,j}^n< l-T^*_{v,j}+\epsilon.$$
\end{proof}
\end{lemma}

\begin{lemma}\label{antncon}Condition on $\{W^*>0\}$. Suppose that $(a_n)$ satisfies the conditions of Lemma \ref{consmt}. Then, almost surely,
\[
\lim_{n\rightarrow\infty}\left(\mathcal{X}^n(a_n)-Z^n(a_n)e_u\right)=\sum_{i\in\mathcal{S}}\sum_{v\in\mathcal{G}_{i}}e_v\sum_{j=1}^{K^*_v(l)}Y_{v,j}(l-T^*_{v,j}),
\]
where the $Y_{v,j}(\cdot)$ are from Lemma \ref{easys} and the $K_v^*(\cdot)$ and $T_{v,j}^*$ are from Lemma \ref{consmt}.
\begin{remark}By definition, $T^*_{v,j}\leq l$ for $j=1,..,K_v^*(l)$. So the limit in Lemma \ref{antncon} is well defined.
\end{remark}
\begin{proof}[Proof of Lemma \ref{antncon}]Recall that
\bq\label{atcep}
\mathcal{X}^n(a_n)-Z^n(a_n)e_u&=&\sum_{v\in\mathcal{G}\backslash\{u\}}\sum_{j=1}^{K^n_v(a_n)}\mathcal{Y}^n_{v,j}(a_n-T^n_{v,j})\\&&+\sum_{v,w\in\mathcal{G}\backslash\{u\}}\sum_{j=1}^{K^n_{v,w}(a_n)}\mathcal{Y}^n_{v,w,j}(a_n-T^n_{v,w,j}).\nonumber
\eq
The `double mutation' term in (\ref{atcep}) converges to zero, because
\[
K^n_{v,w}(a_n)\leq K^n_{v,w}(\tau_n),
\]
which converges to zero by Lemma \ref{kyztz}. As for the `single mutation' term in (\ref{atcep}), Lemma \ref{consmt} says that the $K^n_v(a_n)$ and $a_n-T^n_{v,j}$ converge to $K^*_v(l)$ and $l-T^*_{v,j}$, while  (\ref{ascmc}) says that the $\mathcal{Y}_{v,j}^n(\cdot)$ converge to $e_vY_{v,j}(\cdot)$.
\end{proof}
\end{lemma}

%
%
%
%

\begin{lemma}\label{tnrncon}Condition on $\{W^*>0\}$.
\[
\lim_{n\rightarrow\infty}(\sigma_n-\tau_n)=0
\]
almost surely.
\begin{proof}
By Lemma \ref{cwd}, for large enough $n$, $Z^n(\tau_n)=n$. So $|\mathcal{X}^n(\tau_n)|\geq n$, and hence $\sigma_n\leq\tau_n$. Therefore
\[
\liminf_{n\rightarrow\infty}(\sigma_n-\tau_n)\leq0.
\]
Suppose, looking for a contradiction, that
\[
\liminf_{n\rightarrow\infty}(\sigma_n-\tau_n)=l\in[-\infty,0).
\]
Take a subsequence with
\[
\lim_{k\rightarrow\infty}(\sigma_{n_k}-\tau_{n_k})=l.
\]
Then, by Lemma \ref{antncon},
\[
|\mathcal{X}^n(\sigma_{n_k})-Z^n(\sigma_{n_k})e_u|
\]
converges, and so must be a bounded sequence. However it is also true that, taking $k\rightarrow\infty$,
\ba
|\mathcal{X}^n(\sigma_{n_k})-Z^n(\sigma_{n_k})e_u|&=&n_k-Z^{n_k}(\sigma_{n_k})\\
&=&n_k\left(1-\frac{Z^{n_k}(\sigma_{n_k})}{e^{\lambda_{n_k}\sigma_{n_k}}}\frac{e^{\lambda_{n_k}\tau_{n_k}}}{Z^{n_k}(\tau_{n_k})}e^{\lambda_{n_k}(\sigma_{n_k}-\tau_{n_k})}\right)\\
&\sim&n_k(1-e^{\lambda l}),
\ea
which is unbounded.
\end{proof}
\end{lemma}

\begin{lemma}\label{coxr}Condition on $\{W^*>0\}$. 
\ba
\lim_{n\rightarrow\infty}(\mathcal{X}^n(\sigma_n)-Z^n(\sigma_n)e_u)&=&\sum_{i\in\mathcal{S}}\sum_{v\in\mathcal{G}_{i}}e_v\sum_{j=1}^{K^*_v(0)}Y_{v,j}(-T^*_{v,j}),
\ea
almost surely.
\begin{proof}By Lemma \ref{tnrncon}, $(\sigma_n)=(a_n)$ satisfies the conditions of Lemma \ref{consmt} with $l=0$. Then Lemma \ref{antncon} gives the result.
\end{proof}
\end{lemma}
Let's look at the limit in Lemma \ref{coxr}. For $v\in\mathcal{G}_i$, $K^*_v(0)$ is Poisson distributed with mean $2\lambda^{-1}\alpha(u')\theta_i^{u_i,v_i}$. Conditional on $K^*_v(0)$, the times $(-T^*_{v,j})_{j=1}^{K^*_v(0)}$, unordered, are i.i.d. exponentially distributed random variables with mean $\lambda^{-1}$. So
$$
\sum_{j=1}^{K^*_v(0)}Y_{v,j}(-T^*_{v,j})
$$
has generalised Luria-Delbr{\"u}ck distribution with parameters
$$
\left(\lambda,\alpha(v'),\beta(v'),2\lambda^{-1}\alpha(u')\theta_i^{u_i,v_i}\right).
$$
Therefore the limit of Lemma \ref{coxr} is a vector of independent generalised Luria-Delbr{\"u}ck distributions:
\[
\sum_{i\in\mathcal{S}}\sum_{v\in\mathcal{G}_{i}}e_v\sum_{j=1}^{K^*_v(0)}Y_{v,j}(-T^*_{v,j})\overset{d}{=}\sum_{v\in\mathcal{G}\backslash\{u\}}e_vX_v,
\]
where the $X_v$ are as stated in Theorem \ref{mt}. To complete the proof of Theorem \ref{mt} we need to show that conditioning on $\{W^*>0\}$ can be translated to conditioning on $\{\sigma_n<\infty\}$, which is the subject of the next subsection.
\subsection{Conditioning on reaching $n$ cells}\label{wstarsigma} In order to connect $\{W^*>0\}$ and $\{\sigma_n<\infty\}$, the next result is the key. It states that these events are approximately the same for large $n$.
\begin{prop}\label{stwmp}~
\begin{enumerate}
\item$\lim_{n\rightarrow\infty}\mathbb{P}[W^*>0,\sigma_n=\infty]=0$, and
\item
$\lim_{n\rightarrow\infty}\mathbb{P}[W^*=0,\sigma_n<\infty]=0$.
\end{enumerate}
\end{prop}
Let's break the proof of Proposition \ref{stwmp} into several lemmas; the idea is that the random variable $W^n$ be used as an intermediary.
\begin{lemma}\label{wwnz}
\[
\lim_{n\rightarrow\infty}\mathbb{P}[W^*>0, W^n=0]=0.
\]
\begin{proof}
If $W^*>0$, then there exists $n_0$, such that for all $n\geq n_0$
\[
W^n>\frac{W^*}{2}.
\]
So
\[
\lim_{n\rightarrow\infty}1_{\{W^*>0, W^n=0\}}= 0.
\]
Therefore, by dominated convergence,
\[
\mathbb{P}[W^*>0, W^n=0]=\mathbb{E} 1_{\{W^*>0, W^n=0\}}\rightarrow 0.
\]

\end{proof}
\end{lemma}
\begin{lemma}\label{wnsi}
\[
\mathbb{P}[W^n>0, \sigma_n=\infty]=0.
\]
\begin{proof}
If $W^n>0$, then $\lim_{t\rightarrow\infty}X^n(t)=\infty$, and so $\sigma_n<\infty$.
\end{proof}
\end{lemma}
\begin{proof}[Proof of Part 1 of Proposition \ref{stwmp}]
\ba
\mathbb{P}[W^*>0,\sigma_n=\infty]&=&\mathbb{P}[W^*>0,\sigma_n=\infty,W^n=0]\\&&+\mathbb{P}[W^*>0,\sigma_n=\infty,W^n>0]\\
&\leq&\mathbb{P}[W^*>0,W^n=0]\\&&+\mathbb{P}[\sigma_n=\infty,W^n>0]\rightarrow0
\ea
as $n\rightarrow\infty$, by Lemmas \ref{wwnz} and \ref{wnsi}.
\end{proof}
The structure for the proof of Part 2 of Proposition \ref{stwmp} is much the same as that of Part 1. However the details will require a little extra work.
\begin{lemma}\label{wzwn}
\[
\lim_{n\rightarrow\infty}\mathbb{P}[W^*=0, W^n>0]=0.
\]
\begin{proof}
Let $\epsilon>0$. If $W^*=0$, then there exists $n_0$ such that for all $n\geq n_0$
\[
W^n<\epsilon.
\]
So
\[
\lim_{n\rightarrow\infty}1_{\{W^*=0, W^n\geq\epsilon\}}=0
\]
almost surely. Then by dominated convergence,
$$
\lim_{n\rightarrow\infty}\mathbb{P}[W^*=0, W^n\geq\epsilon]=0.
$$Meanwhile for each $n$,
$$
\mathbb{P}[W^n\in(0,\epsilon)]=\frac{\lambda_n}{\alpha_n}\left(1-e^{-\frac{\lambda_n}{\alpha_n}\epsilon}\right)\leq\epsilon
$$
(the distribution of $W^n$ is seen in \cite{an,ca}). Therefore
\ba
&&\limsup_{n\rightarrow\infty}\mathbb{P}[W^*=0, W^n>0]\\&&\leq\limsup_{n\rightarrow\infty}\mathbb{P}[W^*=0, W^n\geq\epsilon]+\limsup_{n\rightarrow\infty}\mathbb{P}[W^n\in(0,\epsilon)]\\
&&\leq\epsilon.
\ea
But $\epsilon>0$ was arbitrary, giving the result.
\end{proof}
\end{lemma}

\begin{lemma}\label{wnzs}
\[
\lim_{n\rightarrow\infty}\mathbb{P}[W^n=0,\sigma_n<\infty]=0.
\]
\begin{proof}If the primary population size never reaches $n$ and there are never any mutations, then the total population size never reaches $n$. That is, if $Z^n(\tau_n)=0$, $K_v^n(\cdot)=0$ and $K^n_{v,w}(\cdot)=0$ for all $v,w\in\mathcal{G}\backslash\{u\}$, then
\[
\sup_{t\geq0}|\mathcal{X}^n(t)|<n,
\]
which means that $\sigma_n=\infty$. Equivalently,
\ba
\{\sigma_n<\infty\}&\subset&\{Z^n(\tau_n)=n\}\cup\{\exists v,K_v^n(\cdot)\not=0\}\cup\{\exists (v,w),K_{v,w}^n(\cdot)\not=0\}\\
&=&\{Z^n(\tau_n)=n\}\cup\{\exists v,K_v^n(\cdot)\not=0\}\\
&&\cup\{\exists (v,w),K_{v,w}^n(\cdot)\not=0,Z^n(\tau_n)=0\},
\ea
where the equality relies on the fact that $\{Z^n(\tau_n)=0\}\cup\{Z^n(\tau_n)=n\}$ covers the whole probability space.
It follows that
\bq
\mathbb{P}[W^n=0,\sigma_n<\infty]&\leq&\mathbb{P}[W^n=0|Z^n(\tau_n)=n]\nonumber\\
&&+\sum_{v\in\mathcal{G}\backslash\{u\}}\mathbb{P}[ K^n_v(\cdot)\not=0|W^n=0]\nonumber\\
&&+\sum_{v,w\in\mathcal{G}\backslash\{u\}}\mathbb{P}[ K^n_{v,w}(\cdot)\not=0|Z^n(\tau_n)=0].\label{terms}
\eq
We will show that each term of the right hand side of Inequality (\ref{terms}) converges to zero. Firstly,
\ba
\mathbb{P}[W^n=0|Z^n(\tau_n)=n]&=&\left(\frac{\beta_n}{\alpha_n}\right)^n,
\ea
which is the probability that $Z^n(\cdot)$, if starting at size $n$, eventually goes extinct; this clearly converges to zero.

Secondly,
\ba
&&\mathbb{E}\left[\sup_tK_v^n(t)\Big|W^n=0\right]\\
&&=\mathbb{E}\left[N_v^n\left(2p_n(u,v)\alpha(u')\int_0^\infty Z^n(s)ds\right)\Big|W^n=0\right]\\
&&=\mathbb{E}\bigg[\mathbb{E}\Big[N_v^n\left(2p_n(u,v)\alpha(u')\int_0^\infty Z^n(s)ds\right)\Big|Z^n(\cdot)\Big]\Big|W^n=0\bigg]\\
&&=\mathbb{E}\left[2p_n(u,v)\alpha(u')\int_0^\infty Z^n(s)ds\Big|W^n=0\right]\\
&&=2p_n(u,v)\alpha(u')\int_0^\infty \mathbb{E}[Z^n(s)|W^n=0]ds\\
&&=2p_n(u,v)\alpha(u')\int_0^\infty e^{-\lambda_ns}ds\\
&&\rightarrow0,
\ea
because $p_n(u,v)\rightarrow0$. Hence
\[
\mathbb{P}\left[\sup_tK_v^n(t)\not=0\Big|W^n=0\right]\rightarrow0.
\]
Lastly,
\ba
\mathbb{P}[K^n_{v,w}(\cdot)\not=0|Z^n(\tau_n)=0]&=&\mathbb{P}[K^n_{v,w}(\tau_n)\not=0|Z^n(\tau_n)=0]\\
&\leq& \frac{\mathbb{P}[ K^n_{v,w}(\tau_n)\not=0]}{\mathbb{P}[Z^n(\tau_n)=0]}.
\ea
But $\mathbb{P}[K^n_{v,w}(\tau_n)\not=0]$ converges to zero by Lemma \ref{kyztz}, while $\mathbb{P}[Z^n(\tau_n)=0]$ converges to $\mathbb{P}[W^*=0]>0$ by (\ref{ascpc}).
\end{proof}
\end{lemma}

\begin{proof}[Proof of Part 2 of Proposition \ref{stwmp}]
Just as for Part 1,
\ba
\mathbb{P}[W^*=0,\sigma_n<\infty]&=&\mathbb{P}[W^*=0,\sigma_n<\infty,W^n>0]\\&&+\mathbb{P}[W^*=0,\sigma_n<\infty,W^n=0]\\
&\leq&\mathbb{P}[W^*=0,W^n<0]\\&&+\mathbb{P}[\sigma_n<\infty,W^n=0]\rightarrow0
\ea
as $n\rightarrow\infty$, by Lemmas \ref{wzwn} and \ref{wnzs}.
\end{proof}
\begin{cor}[to Proposition \ref{stwmp}]\label{cortws}For any sequence of events $(H_n)_{n\in\mathbb{N}}$,
\[
\lim_{n\rightarrow\infty}\mathbb{P}[H_n,\sigma_n<\infty]=\lim_{n\rightarrow\infty}\mathbb{P}[H_n,W^*>0]
\]
if the limit exists.
\begin{proof}Partition the event $\{H_n\cap(W^*>0\cup\sigma_n<\infty)\}$ in two ways to obtain
\ba
&&\mathbb{P}[H_n,\sigma_n<\infty]+\mathbb{P}[H_n,W^*>0,\sigma_n=\infty]\\&&=\mathbb{P}[H_n,W^*>0]+\mathbb{P}[H_n,W^*=0,\sigma_n<\infty],
\ea
and take $n\rightarrow\infty$.
\end{proof}
\end{cor}
Finally we are in a position to prove Theorem \ref{mt}.
\begin{proof}[Proof of Theorem \ref{mt}]For any $R\subset(\mathbb{N}_0)^{\mathcal{G}\backslash\{u\}}$,
\bq\label{mtaf}
&&\lim_{n\rightarrow\infty}\frac{\mathbb{P}\left[\left(X^{\mu_n}_v(\sigma_n)\right)_{v\in\mathcal{G}\backslash\{u\}}\in R,\sigma_n<\infty\right]}{\mathbb{P}\left[\sigma_n<\infty\right]}\nonumber\\
&&=\lim_{n\rightarrow\infty}\frac{\mathbb{P}\left[\left(X^{\mu_n}_v(\sigma_n)\right)_{v\in\mathcal{G}\backslash\{u\}}\in R,W^*>0\right]}{\mathbb{P}\left[W^*>0\right]}\label{talc}\\
&&=\mathbb{P}[(X_v)_{v\in\mathcal{G}\backslash\{u\}}\in R],\label{titfr}
\eq
where (\ref{talc}) is due to Corollary \ref{cortws} and (\ref{titfr}) is due to Lemma \ref{coxr}.

\end{proof}
\subsection{Law of large numbers}\label{llnpss}
\begin{proof}[Proof of Theorem \ref{llnws} (and Theorem \ref{llnldd})]
The expected site frequency spectrum is given by
\bq
\mathbb{E}\left[|\mathcal{S}|^{-1}\sum_{i\in\mathcal{S}}\delta_{B_i^{n,\mu}}\{k\}\right]&=&|\mathcal{S}|^{-1}\sum_{i\in\mathcal{S}}\mathbb{P}[B_i^{n,\mu}=k]\nonumber\\
&=&|\mathcal{S}|^{-1}\sum_{i\in\mathcal{S}_\text{sel}}\mathbb{P}[B_i^{n,\mu}=k]\nonumber\\
&&+|\mathcal{S}|^{-1}\sum_{j\in J}\sum_{i\in\mathcal{S}(j)}\mathbb{P}[B_i^{n,\mu}=k],\label{sfsllnpt}
\eq
for $k\in\mathbb{N}_0$. The penultimate term of (\ref{sfsllnpt}) vanishes:
$$|\mathcal{S}|^{-1}\sum_{i\in\mathcal{S}_\text{sel}}\mathbb{P}[B_i^{n,\mu}=k]\rightarrow0$$ because $|\mathcal{S}_\text{sel}|/|\mathcal{S}|\rightarrow0$. The last term of (\ref{sfsllnpt}) can be written as 
$$
|\mathcal{S}|^{-1}\sum_{j\in J}\sum_{i\in\mathcal{S}(j)}\mathbb{P}[B_i^{n,\mu}=k]=\sum_{j\in J}\frac{|\mathcal{S}(j)|}{|\mathcal{S}|}\mathbb{P}[B^{n,\mu}(j)=k],
$$
where $\mathbb{P}[B^{n,\mu}(j)=k]=\mathbb{P}[B^{n,\mu}_i=k]$ for $i\in\mathcal{S}(j)$. But $$\frac{|\mathcal{S}(j)|}{|\mathcal{S}|}\rightarrow q(j),$$ while Theorem \ref{ilddhmrs} implies that
$$\mathbb{P}[B^{n,\mu}(j)=k]\rightarrow\Lambda(j)\{k\}.$$
Therefore the expected site frequency spectrum converges:
$$
\mathbb{E}\left[|\mathcal{S}|^{-1}\sum_{i\in\mathcal{S}}\delta_{B_i^{n,\mu}}\{k\}\right]\rightarrow\sum_{j\in J}q(j)\Lambda(j)\{k\}.
$$
The variance is
\ba
\text{Var}\left[|\mathcal{S}|^{-1}\sum_{i\in\mathcal{S}}\delta_{B_i^{n,\mu}}\{k\}\right]&=&|\mathcal{S}|^{-2}\sum_{i\in\mathcal{S}}\text{Var}[1_{\{B_i^{n,\mu}=k\}}]\\
&&+|\mathcal{S}|^{-2}\sum_{\substack{i,j\in\mathcal{S}\\i\not=j}}\text{Cov}[1_{\{B_i^{n,\mu}=k\}},1_{\{B_j^{n,\mu}=k\}}]\\
&\leq&|\mathcal{S}|^{-1}+\max_{\substack{i,j\in\mathcal{S}\\i\not=j}}\text{Cov}[1_{\{B_i^{n,\mu}=k\}},1_{\{B_j^{n,\mu}=k\}}].
\ea
Because $\mathcal{S}_\text{sel}$ and $J$ are finite sets and the random variables are exchangable over $\mathcal{S}(j)$, the maximum is taken over a finite set. Theorem \ref{ilddhmrs} says that the covariances converge to zero.
\end{proof}
\section{Mutations at positive fractions}\label{sec:mpfps}
In this section we return to the basic Yule process setting, proving results on mutations present in a positive fraction of cells. In Subsection \ref{sec:ssmfaub} we prove Theorem \ref{sslmf} and also prove an upper bound for mutation frequencies.  In \ref{sec:dcfs} we prove Lemma \ref{btst} and another result concerning cell descendant fractions. In \ref{sec:mfisa} we determine mutation frequencies \textit{under the infinite sites assumption}. In \ref{sec:isaia} we show that the infinite sites assumption can offer an approximation for mutation frequencies, concluding the proof of Theorem \ref{dslfm}. In \ref{pfrsfs} we give details of Remarks \ref{msflfl} and \ref{vsflfl}.

\subsection{Single site mutation frequencies and an upper bound}\label{sec:ssmfaub}Write
\[
\mathcal{B}_i^\mu=\{x\in\mathcal{T}:V_i^\mu(x)\not=u_i\}
\]
for the cells which are mutated at site $i\in\mathcal{S}$, and
\[
\hat{\mathcal{B}}_i^\mu=\{y\in\mathcal{T}:\exists x\in\mathcal{B}^\mu_i,x\preceq y\}
\]
for their descendants. Recall that $\mathcal{T}_{\sigma_n}$ are the cells alive when the total number of cells reaches $n$. Note the inequality
\bq\label{bbhatineq}
\hat{B}^{n,\mu}_i:=|\hat{\mathcal{B}}_i^\mu\cap\mathcal{T}_{\sigma_n}|\geq|\mathcal{B}_i^\mu\cap\mathcal{T}_{\sigma_n}|=B^{n,\mu}_i.
\eq
The goal of this subsection is to prove Theorem \ref{sslmf} and the following closely related result, which will later play a crucial role in the proof of Theorem \ref{dslfm}.
\begin{prop}\label{ssmfps}Let $i\in\mathcal{S}$ and $a\in(0,1)$. As $n\rightarrow\infty$ and $n\mu\rightarrow\theta\in[0,\infty)$,
\[
\mu^{-1}\mathbb{P}[n^{-1}\hat{B}_i^{n,\mu}>a]\rightarrow2(a^{-1}-1).
\]
\end{prop}
For this subsection we are always talking about a single site $i\in\mathcal{S}$; for convenience, let's drop the subscript $i$ from the notation. To begin the proofs of Theorem \ref{sslmf} and Proposition \ref{ssmfps}, fix $\mu$, and observe that $(B^{r,\mu})_{r\in\mathbb{N}}$ is a Markov process on the nonnegative integers with transition probabilities
\bq
&&\mathbb{P}[B^{r+1,\mu}=k|B^{r,\mu}=j]\nonumber\\&&\quad=\begin{cases}
\frac{j}{r}(\mu/3)^2,\quad &k=j-1;\\
\frac{j}{r}2(\mu/3)(1-\mu/3)+\frac{r-j}{r}(1-\mu)^2,\quad &k=j;\\
\frac{j}{r}(1-\mu/3)^2+\frac{r-j}{r}2\mu(1-\mu),\quad &k=j+1;\\
\frac{r-j}{r}\mu^2,\quad &k=j+2.
\end{cases}\label{transprobs}
\eq
Here, $j/r$ is the probability that one of the $j$ mutant cells divides, while $\mu/3$ is the probability that a mutant's daughter reverts to the unmutated state. The process $(\hat{B}^{r,\mu})_{r\in\mathbb{N}}$ has transition probabilities
\bq
&&\mathbb{P}[\hat{B}^{r+1,\mu}=k|\hat{B}^{r,\mu}=j]\nonumber\\&&\quad=\begin{cases}
\frac{r-j}{r}(1-\mu)^2,\quad &k=j;\\
\frac{j}{r}+\frac{r-j}{r}2\mu(1-\mu),\quad &k=j+1;\\
\frac{r-j}{r}\mu^2,\quad &k=j+2.
\end{cases}\label{transprobsb}
\eq
 The key idea of the proof will be to condition on the number of cells when the first mutant (with respect to site $i$) arises. For this purpose, introduce
\[
\xi^\mu=\min\{r\in\mathbb{N}:B^{r,\mu}>0\}
\]
for the total number of cells when the first mutant cell arises. Let
\[
\Xi^\mu_j=\{B^{\xi^\mu,\mu}=j\}
\]
be the event that the first cell division to see a mutation gives $j$ mutant cells, for $j\in\{1,2\}$.
\begin{lemma}\label{fmdl}For $r\in\mathbb{N}$,
\ba
\lim_{\mu\rightarrow0}\mu^{-1}\mathbb{P}[\xi^{\mu}=r,\Xi^{\mu}_j]=\begin{cases}2,\quad j=1;\\0,\quad j=2.\end{cases}
\ea
\end{lemma}
\begin{proof}The probability that the first $r-2$ cell divisions give no site $i$ mutations multiplied by the probability that the $(r-1)$th cell division gives exactly one mutant daughter is
\[
\mathbb{P}[\xi^{\mu}=r,\Xi^{\mu}_1]=(1-\mu)^{2r-3}2\mu.
\]
Similarly
\[
\mathbb{P}[\xi^{\mu}=r,\Xi^{\mu}_2]=(1-\mu)^{2r-4}\mu^2.
\]
Divide by $\mu$ and take $\mu\rightarrow0$.
\end{proof}
The next result gives conditional mutation frequencies.
\begin{lemma}\label{ck1l}Let $a>0$. As $n\rightarrow\infty$ and $n\mu\rightarrow\theta\in[0,\infty)$,
\[
\mathbb{P}[n^{-1}B^{n,\mu}>a\big|\xi^{\mu}=r,\Xi^{\mu}_1]\rightarrow(1-a)^{r-1}
\]
and
\[
\mathbb{P}[n^{-1}\hat{B}^{n,\mu}>a\big|\xi^{\mu}=r,\Xi^{\mu}_1]\rightarrow(1-a)^{r-1}
\]
\begin{proof}
Calculating from the transition probabilities (\ref{transprobs}),
\ba
\mathbb{E}\left[B^{s+1,\mu}|B^{s,\mu}=k\right]=k+s^{-1}k(1-8\mu/3)+2\mu.
\ea
So
\ba
\mathbb{E}\left[(s+1)^{-1}B^{s+1,\mu}|B^{s,\mu}=k\right]=s^{-1}k+ s^{-1}(s+1)^{-1}\left(2s-\frac{8}{3}k\right)\mu,
\ea
and hence
\bq\label{expbound}
s^{-1}k-2s^{-1}\mu&\leq&\mathbb{E}\left[(s+1)^{-1}B^{s+1,\mu}|B^{s,\mu}=k\right]\nonumber\\
&\leq& s^{-1}k+2s^{-1}\mu.
\eq
For the rest of the proof we will condition on the event $\{\xi^{\mu}=r,\Xi^{\mu}_1\}$. That is, we will consider the processes $(B^{s,\mu})_{s\geq r}$ and $(\hat{B}^{s,\mu})_{s\geq r}$ conditioned on $B^{r,\mu}=\hat{B}^{r,\mu}=1$. Write
$$\mathbb{E}_{r,1}[\cdot]=\mathbb{E}[\cdot|\xi^{\mu}=r,\Xi^{\mu}_1]$$
for the conditional expectation. From (\ref{expbound}), for $s\geq r$,

\bq\label{elubt}
\mathbb{E}_{r,1}\left[s^{-1}B^{s,\mu}\right]-2s^{-1}\mu&\leq&\mathbb{E}_{r,1}\left[(s+1)^{-1}B^{s+1,\mu}\right]\nonumber\\
&\leq&\mathbb{E}_{r,1}\left[s^{-1}B^{s,\mu}\right]+2s^{-1}\mu.
\eq
Combining (\ref{elubt}) with
\[
\mathbb{E}_{r,1}\left[r^{-1}B^{r,\mu}\right]=r^{-1},
\]
we have that, for $n\geq r$,
\ba
r^{-1}-2\mu\sum_{s=r}^{n-1}s^{-1}\leq\mathbb{E}_{r,1}\left[n^{-1}B^{n,\mu}\right]\leq r^{-1}+2\mu\sum_{s=r}^{n-1}s^{-1}.
\ea
Therefore, as $\mu\rightarrow0$ and $n\mu\rightarrow\theta$,
\ba
\mathbb{E}_{r,1}\left[n^{-1}B^{n,\mu}\right]\rightarrow r^{-1}.
\ea
In just the same manner,
\ba
\mathbb{E}_{r,1}\left[n^{-1}\hat{B}^{n,\mu}\right]\rightarrow r^{-1}.
\ea
Consider the single mutant cell present when the total number of cells is $r$. Write $D^n$ for the number of cells which have descended from this mutant cell when the total number of cells is $n\geq r$. The process $(D^n,n-D^n)_{n\geq r}$ is just Polya's urn. So
\ba
\mathbb{E}_{r,1}\left[n^{-1}D^n\right]=r^{-1}.\ea
Moreover a well-known result (e.g. \cite{durr}) says that, as $n\rightarrow\infty$, $n^{-1}D^n$ converges to a Beta random variable. That is, for $a\in(0,1)$,
\bq\label{betalim}
\mathbb{P}_{r,1}[n^{-1}D^n>a]\rightarrow(1-a)^{r-1},
\eq
which is exactly the limit we wish to show for $n^{-1}B^{n,\mu}$ and $n^{-1}\hat{B}^{n,\mu}$. To show that $n^{-1}D^n$, $n^{-1}B^{n,\mu}$, and $n^{-1}\hat{B}^{n,\mu}$ share the same limiting distribution, we will show that their differences converge to zero. The inequality
\[
D^n\leq\hat{B}^{n,\mu}
\]
gives that
\[
\mathbb{E}_{r,1}|n^{-1}\hat{B}^{n,\mu}-n^{-1}D^n|=\mathbb{E}_{r,1}n^{-1}\hat{B}^{n,\mu}-\mathbb{E}_{r,1}n^{-1}D^n\rightarrow0.
\]
The inequality
\[
B^{n,\mu}\leq\hat{B}^{n,\mu}
\]
gives that
\[
\mathbb{E}_{r,1}|n^{-1}\hat{B}^{n,\mu}-n^{-1}B^{n,\mu}|=\mathbb{E}_{r,1}n^{-1}\hat{B}^{n,\mu}-\mathbb{E}_{r,1}n^{-1}B^{n,\mu}\rightarrow0.
\]
\end{proof}
\end{lemma}
\begin{lemma}\label{df}Consider $(\mu_n)_{n\in\mathbb{N}}$ with $n\mu_n\rightarrow\theta\in[0,\infty)$. Then
\ba
\sup_{n\in\mathbb{N}}\mathbb{P}[n^{-1}B^{n,\mu_n}>a\big|\xi^{\mu_n}=r,\Xi^{\mu_n}_j]&\leq&\sup_{n\in\mathbb{N}}\mathbb{P}[n^{-1}\hat{B}^{n,\mu_n}>a\big|\xi^{\mu_n}=r,\Xi^{\mu_n}_j]\\
&\leq& cr^{-2},
\ea
where $c>0$ does not depend on $r,j$.
\end{lemma}
\begin{proof}
The first inequality is immediate. We prove the second. From the transition probabiities (\ref{transprobsb}),
\ba
\mathbb{E}\left[(\hat{B}^{s+1,\mu_n})^2|\hat{B}^{s,\mu_n}=k\right]&=&k^2\left(1+2s^{-1}-4\mu_n s^{-1}\right)\\
&&+k\left(4\mu_n+s^{-1}(1-2\mu_n-2\mu_n^2)\right)\\
&&+2\mu_n+2\mu_n^2\\
&\leq&k^2(s+1)^2s^{-2}+4n\mu_n+2\mu_n+2\mu_n^2+1.
\ea
So, for $s\leq n$,
\bq\label{expineq}
\mathbb{E}\left[(s+1)^{-2}(\hat{B}^{s+1,\mu_n})^2|\hat{B}^{s,\mu_n}=k\right]\leq s^{-2}k^2+s^{-2}c_1,
\eq
where $c_1>0$ is a constant which does not depend on $k,n,s$. Now let's condition on $\{\xi^{\mu_n}=r,\Xi^{\mu_n}_j\}$, again writing $\mathbb{E}_{r,j}$ for the conditional expectation. From (\ref{expineq}), for $s\in\{r,..,n\}$,
\[
\mathbb{E}_{r,j}[(s+1)^{-2}(\hat{B}^{s+1,\mu_n})^2]-\mathbb{E}_{r,j}[s^{-2}(\hat{B}^{s,\mu_n})^2]\leq s^{-2}c_1.
\]
This leads to, for $s\in\{r,..,n\}$,
\bq\label{secondmbound}
\mathbb{E}_{r,j}[s^{-2}(\hat{B}^{s,\mu_n})^2]&\leq&\mathbb{E}_{r,j}[r^{-2}(\hat{B}^{r,\mu_n})^2]+c_1\sum_{t=r}^{s-1}t^{-2}\nonumber\\
&=&r^{-2}j^2+c_1\sum_{t=r}^{s-1}s^{-2}\nonumber\\
&\leq&c_2r^{-1},
\eq
where $c_2$ is a constant which does not depend on $j,n,r,s$. (In the following, $c_3,..,c_6$ will also be constants.) Calculating third moments from the transition probabilities,
\ba
\mathbb{E}\left[(\hat{B}^{s+1,\mu_n})^3|\hat{B}^{s,\mu_n}=k\right]&=&k^3\left(1+3s^{-1}(1-2\mu_n-4\mu_n^2)\right)\\
&&+k^2\left(6\mu_n+3s^{-1}(1-2\mu_n-6\mu_n^2)\right)\\
&&+k\left(6\mu_n+6\mu_n^2-s^{-1}(1+2\mu_n+6\mu_n^2)\right)\\
&&+2\mu_n+6\mu_n^2\\
&\leq&k^3(s+1)^{3}s^{-3}+c_3k^2s^{-1}+c_4kn^{-1}.
\ea
Then
\ba
\mathbb{E}\left[(s+1)^{-3}(\hat{B}^{s+1,\mu_n})^3|\hat{B}^{s,\mu_n}=k\right]&\leq&s^{-3}k^3+c_3s^{-4}k^2+c_4s^{-2}n^{-1}.
\ea
Hence
\ba
&&\mathbb{E}_{r,j}\left[(s+1)^{-3}(\hat{B}^{s+1,\mu_n})^3\right]-\mathbb{E}_{r,j}\left[s^{-3}(\hat{B}^{s,\mu_n})^3\right]\\&&\quad\leq c_3s^{-2}\mathbb{E}_{r,j}[s^{-2}(\hat{B}^{s,\mu_n})^2]+c_4s^{-2}r^{-1},
\ea
which combined with (\ref{secondmbound}) gives that
\ba
\mathbb{E}_{r,j}[n^{-3}(\hat{B}^{n,\mu_n})^3]&=&\mathbb{E}_{r,j}[r^{-3}(\hat{B}^{r,\mu_n})^3]+c_5r^{-1}\sum_{s=r}^{n-1} s^{-2}\\
&\leq&r^{-3}+c_6r^{-2}\\
&\leq&cr^{-2}.
\ea
Apply Markov's inequality to conclude.
\end{proof}

\begin{proof}[Proof of Theorem \ref{sslmf} and Proposition \ref{ssmfps}]
\bq
&&\mu^{-1}\mathbb{P}[n^{-1}B^{n,\mu}>a]\nonumber\\
&&=\sum_{r=2}^\infty\sum_{j=1}^2\mu^{-1}\mathbb{P}[\xi^\mu=r,\Xi^\mu_j]\mathbb{P}[n^{-1}B^{n,\mu}>a\big|\xi^\mu=r,\Xi^\mu_j].\label{itwwc}
\eq
Lemmas \ref{fmdl}, \ref{ck1l}, and \ref{df}, along with the Dominated Convergence Theorem, show that the limit of (\ref{itwwc}) is
\ba2\sum_{r=1}^\infty(1-a)^{r}=2(a^{-1}-1).
\ea
The same argument works for $\hat{B}^{n,\mu}$.
\end{proof}

\subsection{Cell descendant fractions}\label{sec:dcfs}
Here we are concerned with the $P_{x,t}$ (the fraction of cells alive at time $t\geq0$ which are descendants of cell $x\in\mathcal{T}$).

Aldous \cite{apdoc}, in a different language to ours, gives a similar result to Lemma \ref{btst}. Rather than adapting his result, we now give a distinct proof of Lemma \ref{btst}.
\begin{proof}[Proof of Lemma \ref{btst}]
Write
\[
\mathcal{D}_x=\{y\in\mathcal{T}:x\preceq y\}
\]
for the descendants of cell $x\in\mathcal{T}$, and write
\[
\mathcal{D}_{x,t}=\mathcal{D}_x\cap\mathcal{T}_t
\]
for the descendants of cell $x\in\mathcal{T}$ which are alive at time $t\geq0$. Observe that
\ba
\mathcal{D}_{x,\sum_{y\prec x}A_y+t}=\left\{y\in\mathcal{T}:x\preceq y, \sum_{x\preceq z\prec y}A_z\leq t<\sum_{x\preceq z\preceq y}A_z\right\}.
\ea
Hence
\[
\left(|\mathcal{D}_{x,\sum_{y\prec x}A_y+t}|\right)_{t\geq0}
\]
is measurable with respect to the sigma-algebra generated by $(A_y)_{y\in\mathcal{D}_x}$, and has the same distribution as
\[
\left(|\mathcal{D}_{\emptyset,t}|\right)_{t\geq0}=\left(|\mathcal{T}_t|\right)_{t\geq0}.
\]
It follows that
\[
\lim_{t\rightarrow\infty}e^{- t}|\mathcal{D}_{x,\sum_{y\prec x}A_y+t}|=:W_x
\]
almost surely, where $W_x\sim$Exp$(1)$; moreover if $x,y\in\mathcal{T}$ are such that $\mathcal{D}_x\cap\mathcal{D}_y=\emptyset$, then $W_x$ and $W_y$ are independent. In particular, $W_{x0}$ and $W_{x1}$ are independent. Now,
\[
\lim_{t\rightarrow\infty}\frac{|\mathcal{D}_{x0,t}|}{|\mathcal{D}_{x,t}|}=\lim_{t\rightarrow\infty}\frac{|\mathcal{D}_{x0,t}|}{1+|\mathcal{D}_{x0,t}|+|\mathcal{D}_{x1,t}|}=\frac{W_{x0}}{W_{x0}+W_{x1}}=:U_{x0}
\]
almost surely, and $U_{x0}+U_{x1}=1$. A standard calculation shows that $U_{x0}$ is uniformly distributed on $(0,1)$: for $u\in(0,1)$,
\[
\mathbb{P}[U_{x0}<u]=\int_0^\infty\int^\infty_{z(1-u)/u} e^{-y}e^{-z}dydz=u.
\]
It remains to show independence of the $U_{x0}$. Another standard calculation shows that
\[
U_{x0}=\frac{W_{x0}}{W_{x0}+W_{x1}}
\]
is independent of
\[
W_{x0}+W_{x1}:
\]
for $(u,v)\in(0,1)\times(0,\infty)$,
\ba
\mathbb{P}[U_{x0}<u,W_{x0}+W_{x1}<v]&=&\int_0^{uv}\int_{z(1-u)/u}^{v-z}e^{-y}e^{-z}dydz\\
&=&u(1-(1+v)e^{-v})\\
&=&\mathbb{P}[U_{x0}<u]\mathbb{P}[W_{x0}+W_{x1}<v].
\ea

Now fix $l\in\mathbb{N}$. Because $U_{x0}$ and $W_{x0}+W_{x1}$ are measurable with respect to the sigma-algebra generated by $(A_y)_{y\in\mathcal{D}_x\backslash\{x\}}$, we have that
\bq\label{cirv}
\left[(U_{x0})_{|x|=l},(W_{x0}+W_{x1})_{|x|=l},(A_x)_{|x|\leq l}\right]
\eq
forms an independent family of random variables.

Finally we complete the proof by induction. Suppose that $(U_{x0})_{x\in\mathcal{T}:|x|< l}$ is an independent family. Observing that for any $x\in\mathcal{T}$,
\[
W_x=e^{-A_x}(W_{x0}+W_{x1}),
\]
we have that $(U_{x0})_{x\in\mathcal{T}:|x|< l}$ is measurable with respect to the sigma-algebra generated by
\[
[(W_{x0}+W_{x1})_{|x|=l},(A_x)_{|x|\leq l}].
\]
Then, thanks to the independence of (\ref{cirv}), $(U_{x0})_{x\in\mathcal{T}:|x|\leq l}$ forms an independent family.\end{proof}
Next comes a technical result whose value will become apparent in the next subsection.
\begin{lemma}\label{finitepset}
Let $\epsilon\in(0,1)$. The set
\[
\{x\in\mathcal{T}\backslash\{\emptyset\}:\exists t\geq0,P_{x,t}>\epsilon\}
\]
is almost surely finite.
\begin{proof}
For $t\geq0$, let $\mathcal{F}_t$ be the sigma-algebra generated by $\left(\mathcal{T}_s\right)_{s\in[0,t]}$. For $t\geq s\geq0$, conditional on $\mathcal{F}_s$, $\left(P_{y,t}\right)_{y\in\mathcal{T}_s}$ is exchangable. So for $y\in\mathcal{T}$, $$\mathbb{E}[1_{\{y\in\mathcal{T}_s\}}P_{y,t}|\mathcal{F}_s]=\frac{1_{\{y\in\mathcal{T}_s\}}}{|\mathcal{T}_s|}.$$
Now let $x\in\mathcal{T}\backslash\{\emptyset\}$. We have
$$
P_{x,t}\geq1_{\{|\mathcal{D}_{x,s}|>0\}}P_{x,t}=\sum_{y\in\mathcal{D}_{x,s}}P_{y,t}$$
and hence
$$
\mathbb{E}[P_{x,t}|\mathcal{F}_s]\geq P_{x,s}.
$$
That is, $(P_{x,t})_{t\geq0}$ is a submartingale with respect to $(\mathcal{F}_t)_{t\geq0}$.
Then by Doob's inequality,
\[
\mathbb{P}[\exists t\geq0,P_{x,t}>\epsilon]=\mathbb{P}[\sup_{t\geq0}P_{x,t}>\epsilon]\leq \epsilon^{-2}\mathbb{E}[(P_x)^2].
\]
But $P_x$ is simply a product of $|x|$ independent Uniform$(0,1)$ random variables, where $|x|\in\mathbb{N}$ is the generation of $x$ (that is, $x\in\{0,1\}^{|x|}$). So $\mathbb{E}[(P_x)^2]=3^{-|x|}$. Hence
\[
\mathbb{P}[\exists t\geq0,P_{x,t}>\epsilon]=\mathbb{P}[\sup_{t\geq0}P_{x,t}>\epsilon]\leq \epsilon^{-2}3^{-|x|}.
\]
Now
\[
\sum_{x\in\mathcal{T}\backslash\{\emptyset\}}\mathbb{P}[\exists t\geq0,P_{x,t}>\epsilon]\leq\sum_{l\in\mathbb{N}}\epsilon^{-2}(2/3)^{l}<\infty,
\]
so the Borel-Cantelli lemma concludes the proof.
\end{proof}
\end{lemma}

\subsection{Mutation frequencies under the infinite sites assumption}\label{sec:mfisa}
Enumerate the elements of $\mathcal{T}$,
\[
\mathcal{T}=(x_k)_{k\in\mathbb{N}},
\]
in such a way that 
\bq\label{propertyofenumeration}x_j\prec x_k\implies j<k.
\eq
Let's give an example of such an enumeration: map $(x(r))_{r=1}^l\in\{0,1\}^l\subset\mathcal{T}$ to $2^l+\sum_{r=1}^l2^{l-r}x(r)$.

Assuming a mutation rate $\mu$, write
\[
\phi_i^\mu=\min\{x\in\mathcal{T}:V^\mu_i(x)\not=u_i\}
\]
for the first cell (with respect to the enumeration) which sees a mutation at site $i\in\mathcal{S}$.
\begin{remark}\label{phiremark}$\phi_i^\mu$ has geometric distribution:
\[\mathbb{P}[\phi_i^\mu=x_k]=\mu(1-\mu)^{k-1}.
\]
\end{remark}
In this subsection we are concerned with $P_{\phi_i^\mu,\sigma_n}$, which is the fraction of cells alive at time $\sigma_n$ (when $n$ total cells are reached) which are descendants of cell $\phi_i^\mu$. Phrased another way, $P_{\phi_i^\mu,\sigma_n}$ is the fraction of cells alive at time $\sigma_n$ which are mutated at site $i$ \textit{under the infinite sites assumption}.

Next we give an infinite-sites analog of Theorem \ref{dslfm}.
\begin{prop}\label{isvmt}As $n\rightarrow\infty$, $\mu\rightarrow0$, and $|\mathcal{S}|\mu\rightarrow\eta\in[0,\infty)$,
\[
\sum_{i\in\mathcal{S}}\delta_{P_{\phi_i^n,\sigma_n}}\rightarrow\sum_{x\in\mathcal{T}\backslash\{\emptyset\}}M_x\delta_{P_x}
\]
in distribution, with respect to the vague topology on the space of measures on $(0,1]$.
\end{prop}
The proof of Proposition \ref{isvmt} will require us to count the number of sites which see their first mutation at cell $x\in\mathcal{T}\backslash\{\emptyset\}$ (with respect to the enumeration); write
\[
M^{\mu,\mathcal{S}}_x=|\{i\in\mathcal{S}:\phi_i^\mu=x\}|.
\] 
\begin{lemma}\label{cotnmn}As $\mu\rightarrow0$ and $|\mathcal{S}|\mu\rightarrow\eta\in[0,\infty)$,
\[
(M^{\mu,\mathcal{S}}_x)_{x\in\mathcal{T}\backslash\{\emptyset\}}\rightarrow(M_x)_{x\in\mathcal{T}\backslash\{\emptyset\}}
\]
in distribution, where the $M_x$ are i.i.d. Poisson($\eta$) random variables.
\begin{proof}
The initial cell is $x_1=\emptyset$. The number of sites which mutate in cell $x_2$, $M_{x_2}^{\mu,\mathcal{S}}$, is binomially distributed with parameters $\mathcal{S}$ and $\mu$. This converges to a Poisson($\eta$) random variable. Now, for induction, suppose that 
\[
\lim_{n\rightarrow\infty}\left(M_{x_j}^{\mu,\mathcal{S}}\right)_{j=2}^k=\left(M_{x_j}\right)_{j=2}^k
\]
in distribution, where the $M_x$ are i.i.d. Poisson($\eta$) random variables. Then
\bq\label{condinden}
\mathbb{P}\left[\left(M_{x_j}^{\mu,\mathcal{S}}\right)_{j=2}^{k+1}=(m_j)_{j=2}^{k+1}\right]&=&\mathbb{P}\left[M_{x_{k+1}}^{\mu,\mathcal{S}}=m_{k+1}\bigg|\left(M_{x_j}^{\mu,\mathcal{S}}\right)_{j=2}^{k}=(m_j)_{j=2}^{k}\right]\nonumber\\
&&\times\mathbb{P}\left[\left(M_{x_j}^{\mu,\mathcal{S}}\right)_{j=2}^{k}=(m_j)_{j=2}^{k}\right].
\eq
Due to the property (\ref{propertyofenumeration}) of the enumeration, $M_{x_{k+1}}^{\mu,\mathcal{S}}$ conditioned on the event $\left(M_{x_j}^{\mu,\mathcal{S}}\right)_{j=2}^{k}=(m_j)_{j=2}^{k}$ is just a binomial random variable with parameters $|\mathcal{S}|-\sum_{j=2}^km_j$ and $\mu$. Therefore (\ref{condinden}) converges as required.
\end{proof}
\end{lemma}
\begin{proof}[Proof of Proposition \ref{isvmt}]
Fix a sequence of sets of sites $(\mathcal{S}_n)_{n\in\mathbb{N}}$  and a sequence of mutation rates $(\mu_n)_{n\in\mathbb{N}}$ with $\mu_n|\mathcal{S}_n|\rightarrow\eta$. Apply Skorokhod's Representation Theorem to Lemma \ref{cotnmn} to obtain random variables $(M_x^n)_{x\in\mathcal{T}\backslash\{\emptyset\},n\in\mathbb{N}}$ and $(M_x')_{x\in\mathcal{T}\backslash\{\emptyset\}}$ which satisfy
\begin{enumerate}
\item$(M_x^n)_{x\in\mathcal{T}\backslash\{\emptyset\}}\overset{d}{=}(M_x^{\mu_n,\mathcal{S}_n})_{x\in\mathcal{T}\backslash\{\emptyset\}}$ for each $n\in\mathbb{N}$;
\item$(M_x')_{x\in\mathcal{T}\backslash\{\emptyset\}}\overset{d}{=}(M_x)_{x\in\mathcal{T}\backslash\{\emptyset\}}$; and
\item$\lim_{n\rightarrow\infty}(M_x^n)_{x\in\mathcal{T}\backslash\{\emptyset\}}=(M'_x)_{x\in\mathcal{T}\backslash\{\emptyset\}}$ almost surely.
\end{enumerate}
Put the $M^n_x,M'_x$ on the same probability space as the $P_{x,t},P_x$ so that the $M^n_x,M'_x$ are independent of the $P_{x,t},P_x$.
Then
\ba
\sum_{i\in\mathcal{S}}\delta_{P_{\phi_i^{\mu},\sigma_n}}&=&\sum_{x\in\mathcal{T}\backslash\{\emptyset\}}M_x^{\mu_n,\mathcal{S}_n}\delta_{P_{x,\sigma_n}}\\
&\overset{d}{=}&\sum_{x\in\mathcal{T}\backslash\{\emptyset\}}M_x^n\delta_{P_{x,\sigma_n}}.
\ea
Let $I_1,..,I_k\subset(0,1]$ be closed intervals. Then
\bq\label{summandcon}
\left(\sum_{i\in\mathcal{S}}\delta_{P_{\phi_i^{\mu},\sigma_n}}(I_j)\right)_{j=1}^k\overset{d}{=}\left(\sum_{x\in\mathcal{T}\backslash\{\emptyset\}}M_x^n\delta_{P_{x,\sigma_n}}(I_j)\right)_{j=1}^k.
\eq
Lemma \ref{btst} says that $P_{x,\sigma_n}$ converges to $P_x$; and $P_x$ does not lie on the boundaries of the $I_j$ with probability one, so the summands of (\ref{summandcon}) converge pointwise. Meanwhile Lemma \ref{finitepset} says that the sum is over a finite subset of $\mathcal{T}\backslash\{\emptyset\}$.
\end{proof}
\begin{lemma}\label{esfsisa}
Let $a\in(0,1)$. As $n\rightarrow\infty$, $\mu\rightarrow0$, and $|\mathcal{S}|\mu\rightarrow\eta\in[0,\infty)$,
\[
\mathbb{E}\left[\sum_{i\in\mathcal{S}}\delta_{P_{\phi_i^n,\sigma_n}}(a,1)\right]\rightarrow2\eta(a^{-1}-1).
\]
\begin{proof}
First,
\ba
\liminf\mathbb{E}\left[\sum_{i\in\mathcal{S}}\delta_{P_{\phi_i^n,\sigma_n}}(a,1)\right]&\geq&\mathbb{E}\left[\lim\sum_{i\in\mathcal{S}}\delta_{P_{\phi_i^n,\sigma_n}}(a,1)\right]\\
&=&\mathbb{E}\left[\sum_{x\in\mathcal{T}\backslash\{\emptyset\}}M_x\delta_{P_x}(a,1)\right]\\
&=&2\eta(a^{-1}-1),
\ea
by Fatou's lemma, Proposition \ref{isvmt}, and Remark \ref{msflfl}. Second,
\ba
\mathbb{E}\left[\sum_{i\in\mathcal{S}}\delta_{P_{\phi_i^n,\sigma_n}}(a,1)\right]&\leq&\mathbb{E}\left[\sum_{i\in\mathcal{S}}\delta_{n^{-1}\hat{B}_i^{n,\mu}}(a,1)\right]\\
&\rightarrow&2\eta(a^{-1}-1),
\ea
by the inequality $P_{\phi_i^n,\sigma_n}\leq n^{-1}\hat{B}_i^{n,\mu}$ and then Proposition \ref{ssmfps}.
\end{proof}
\end{lemma}

\subsection{The infinite sites assumption approximation}\label{sec:isaia}
\begin{lemma}\label{isaifa}Let $a\in(0,1)$. As $n\rightarrow\infty$, $n\mu\rightarrow\theta\in[0,\infty)$, and $|\mathcal{S}|\mu\rightarrow\eta\in[0,\infty)$,
\[
\mathbb{E}\left|\sum_{i\in\mathcal{S}}\delta_{P_{\phi_i^n,\sigma_n}}(a,1)-\sum_{i\in\mathcal{S}}\delta_{n^{-1}B_i^{n,\mu}}(a,1)\right|\rightarrow0.
\]
\begin{proof}The inequalities
\[
n^{-1}B_i^{n,\mu}\leq n^{-1}\hat{B}_i^{n,\mu}\geq P_{\phi_i^\mu,\sigma_n},
\]
imply that
\[
\sum_{i\in\mathcal{S}}\delta_{n^{-1}B_i^{n,\mu}}(a,1)\leq\sum_{i\in\mathcal{S}}\delta_{n^{-1}\hat{B}_i^{n,\mu}}(a,1)\geq\sum_{i\in\mathcal{S}}\delta_{P_{\phi_i^\mu,\sigma_n}}(a,1).
\]
Hence
\ba
&&\mathbb{E}\left|\sum_{i\in\mathcal{S}}\delta_{P_{\phi_i^n,\sigma_n}}(a,1)-\sum_{i\in\mathcal{S}}\delta_{n^{-1}B_i^{n,\mu}}(a,1)\right|\\
&&\leq\mathbb{E}\left|\sum_{i\in\mathcal{S}}\delta_{n^{-1}\hat{B}_i^{n,\mu}}(a,1)-\sum_{i\in\mathcal{S}}\delta_{P_{\phi_i^n,\sigma_n}}(a,1)\right|\\&&\quad\quad+\mathbb{E}\left|\sum_{i\in\mathcal{S}}\delta_{n^{-1}\hat{B}_i^{n,\mu}}(a,1)-\sum_{i\in\mathcal{S}}\delta_{n^{-1}B_i^{n,\mu}}(a,1)\right|\\
&&=2\mathbb{E}\sum_{i\in\mathcal{S}}\delta_{n^{-1}\hat{B}_i^{n,\mu}}(a,1)-\mathbb{E}\sum_{i\in\mathcal{S}}\delta_{P_{\phi_i^n,\sigma_n}}(a,1)-\mathbb{E}\sum_{i\in\mathcal{S}}\delta_{n^{-1}B_i^{n,\mu}}(a,1),
\ea
which converges to zero thanks to Corollary \ref{lmfsfsm}, Proposition \ref{ssmfps}, and Lemma \ref{esfsisa}.
\end{proof}
\end{lemma}

\begin{proof}[Proof of Theorem \ref{dslfm}]
Let $I_1,..,I_k\subset(0,1]$ be closed intervals. Writing $||\cdot||$ for the $l_1$-norm on $\mathbb{R}^k$,
\ba
&&\mathbb{E}\left|\left|\left(\sum_{i\in\mathcal{S}}\delta_{n^{-1}B_i^{n,\mu}}(I_j)\right)_{j=1}^k-\left(\sum_{i\in\mathcal{S}}\delta_{P_{\phi_i^\mu,\sigma_n}}(I_j)\right)_{j=1}^k\right|\right|\\
&&=\sum_{j=1}^k\mathbb{E}\left|\sum_{i\in\mathcal{S}}\delta_{n^{-1}B_i^{n,\mu}}(I_j)-\sum_{i\in\mathcal{S}}\delta_{P_{\phi_i^\mu,\sigma_n}}(I_j)\right|\\
&&\rightarrow0,
\ea
due to Lemma \ref{isaifa}. Therefore
\[
\left(\sum_{i\in\mathcal{S}}\delta_{n^{-1}B_i^{n,\mu}}(I_j)\right)_{j=1}^k
\]
and 
\[
\left(\sum_{i\in\mathcal{S}}\delta_{P_{\phi_i^\mu,\sigma_n}}(I_j)\right)_{j=1}^k
\]
share the same limiting distribution, if it exists. This limiting distribution, by Proposition \ref{isvmt}, is
\[
\left(\sum_{x\in\mathcal{T}\backslash\{\emptyset\}}M_x\delta_{P_x}(I_j)\right)_{j=1}^k
\]
as required.
\end{proof}
\subsection{Mean and variance of the site frequency spectrum}\label{pfrsfs}
Finally, let's check Remarks \ref{msflfl} and \ref{vsflfl}. Note that
\[
-\log(P_x)=-\sum_{y\preceq x}\log(U_y)
\]
is a sum of i.i.d. mean-$1$ exponentially distributed random variables, which is just a gamma random variable with parameters $|x|$ and $1$. Then
\ba
\mathbb{P}[P_x>a]&=&\mathbb{P}[-\log(P_x)<-\log(a)]\\
&=&\int_0^{-\log(a)}\frac{s^{|x|-1}e^{-s}}{(|x|-1)!}ds.
\ea
So
\ba
\mathbb{E}\left[\sum_xM_x\delta_{P_x}(a,1)\right]&=&\eta\sum_x\mathbb{P}[P_x>a]\\
&=&\eta\sum_x\int_0^{-\log(a)}\frac{s^{|x|-1}e^{-s}}{(|x|-1)!}ds\\
&=&\eta\int_0^{-\log(a)}\sum_{l\in\mathbb{N}}2^l\frac{s^{l-1}e^{-s}}{(l-1)!}ds\\
&=&2\eta\int_0^{-\log(a)}e^sds\\
&=&2\eta(a^{-1}-1).
\ea
As for the variance of the site frequency spectrum,
\ba
\text{Var}\left[\sum_xM_x\delta_{P_x}(a,1)\right]&=&\text{Var}\left[\mathbb{E}\left[\sum_xM_x\delta_{P_x}(a,1)\big|(P_x)\right]\right]\\
&&+\mathbb{E}\left[\text{Var}\left[\sum_xM_x\delta_{P_x}(a,1)\big|(P_x)\right]\right]\\
&\geq&\mathbb{E}\left[\text{Var}\left[\sum_xM_x\delta_{P_x}(a,1)\big|(P_x)\right]\right]\\
&=&\mathbb{E}\left[\sum_x\delta_{P_x}(a,1)\text{Var}\left[M_x\right]\right]\\
&=&\eta\mathbb{E}\left[\sum_x\delta_{P_x}(a,1)\right].
\ea

\section{Infinite sites assumption violations}\label{infinitesitessection}The infinite sites assumption (ISA) is a popular modelling assumption, stating that each genetic site can mutate at most once during the population's evolution. There are influential and insightful analyses of tumour evolution which rely on the ISA, for example \cite{gs,ib,matds,tavarerecent}. However, recent statistical analysis of single cell sequencing data shows ``widespread violations of the ISA in human cancers" \cite{isr}. Thus it is unclear to what extent \cite{gs,ib,matds,tavarerecent}'s analyses can be trusted. Our studied model of DNA sequence evolution does not use the ISA and invites a theoretical assessment of the ISA's validity.

Let's check the prevalence of ISA violations. For simplicity, consider the most basic version of the model, which was introduced in Section \ref{sec:model}. Building upon notation of Section \ref{sec:Yuletreesec}, write
\[
\mathcal{T}_{(n)}=\{x\in\mathcal{T}:\exists y\in\mathcal{T}_{\sigma_n},x\prec y\}
\]
for the set of ancestors of those cells alive at time $\sigma_n$ (when the total number of cells reaches $n$). Write
\[
X_i^{n,\mu}=\left|\{(x,xj):x\in\mathcal{T}_{(n)},j\in\{0,1\},V_i^\mu(x)\not=V_i^\mu(xj)\}\right|
\]
for the number of times that site $i$ mutates up to time $\sigma_n$. Observe that $X_i^{n,\mu}$ is binomially distributed with parameters $2n-2$ and $\mu$. Site $i$ is said to violate the ISA if $X_i^{n,\mu}\geq2$, which occurs with probability
\ba
p(n,\mu)&=&\mathbb{P}[X_i^{n,\mu}\geq2]\\
&=&1-(1-\mu)^{2n-2}-(2n-2)\mu(1-\mu)^{2n-3}.
\ea
Then the number of sites to violate the ISA,
\[
\left|\{i\in\mathcal{S}:X_i^{n,\mu}\geq2\}\right|,
\]
is binomially distributed with parameters $|\mathcal{S}|$ and $p(n,\mu)$. If parameter values are indeed in the region of $n=10^9$ and $\mu=10^{-9}$, then the expected proportion of sites to violate the ISA is in the region of $0.5$. This means that the expected number of sites to violate the ISA may be in the billions. Even if very conservative parameter estimates were plugged in, the number of violations is still massive. In fact violations are even more common if one considers cell death. Suppose that cells divide at rate $\alpha$ and die at rate $\beta$. Then to go from a population of $1$ cell to $n$ cells requires approximately $n\alpha/(\alpha-\beta)$ cell divisions, where the factor $\alpha/(\alpha-\beta)$ may be as large as $100$ \cite{ib}. 

Depite the apparent prevalence of ISA violations, our results suggest that their impact on mutation frequencies is negligible at the scale of population fractions. Importantly, bulk sequencing data is only sensitive on the scale of population fractions. Our theoretical work stands in support of the data-driven works \cite{gs,ib,matds,tavarerecent,isr}.

Note however that our model only considers point mutations; it does not, for example, consider deletions of genomic regions, which are thought to be a significant cause of ISA violations \cite{isr}.

\section{Estimating mutation rates}\label{sec:data}
In this section we wish to give the reader a light flavour of mutation frequency data and its relationship to the model. We estimate mutation rates in a lung adenocarcinoma.
\subsection{Diploid perspective}
Before presenting data, an additional ingredient needs to be considered: ploidy. Normal human cells are diploid. That is, chromosomes come in pairs. Therefore a particular mutation may be present zero, one, or two times in a single cell. It should be said that the story is far more complex in tumours, with chromosomal instability and aneuploidy coming into play. Even so, many tumour samples display an average ploidy not so far from two (for example see Figure (1a) of \cite{scna}). We imagine an idealised diploid world.

To illustrate the diploid structure, label the genetic sites as
\[
\mathcal{S}= \{1,2\}\times\{1,..,L\},
\]
for some $L\in\mathbb{N}$. The first coordinate of a site $(i,j)\in\mathcal{S}$ states on which chromosome of a pair the site lies, and the second coordinate refers to the site's position on the chromosome. Mutations at sites $(1,j)$ and $(2,j)$ are typically not distinguished in data. In the original model set up, mutations were defined as differences to the initial cell's genome. Let's slightly improve that definition. Now a genome $v\in\mathcal{G}$ is said to be mutated at site $(i,j)\in\mathcal{S}$ if $v_{i,j}\not= r_j$, where $(r_j)_{j=1}^L$ is some reference. Then data is simplistically stated in the model's language as
\bq\label{mfmtd}
F_j=\frac{1}{2n}\sum_{i=1}^2B^{n,\mu}_{i,j}
\eq
for $j=1,..,L$. That is, the total number of mutations at position $j$ divided by the total number of chromosomes which contain position $j$.
\subsection{A lung adenocarcinoma}
The mutation frequency data of a lung adenocarcinoma was made available in \cite{fdata} (499017, Table S2). The data is plotted in Figures \ref{499017figure} and \ref{datatheoryfigure}. This is just one tumor to illustrate our results. A broader picture of data is seen in \cite{gs,ib}. They analysed hundreds of tumors. Around $1/3$ of the tumors were said to have a power-law distribution for mutation frequencies, resembling the one we consider.



\begin{figure}
\begin{center}
\includegraphics[scale=0.5]{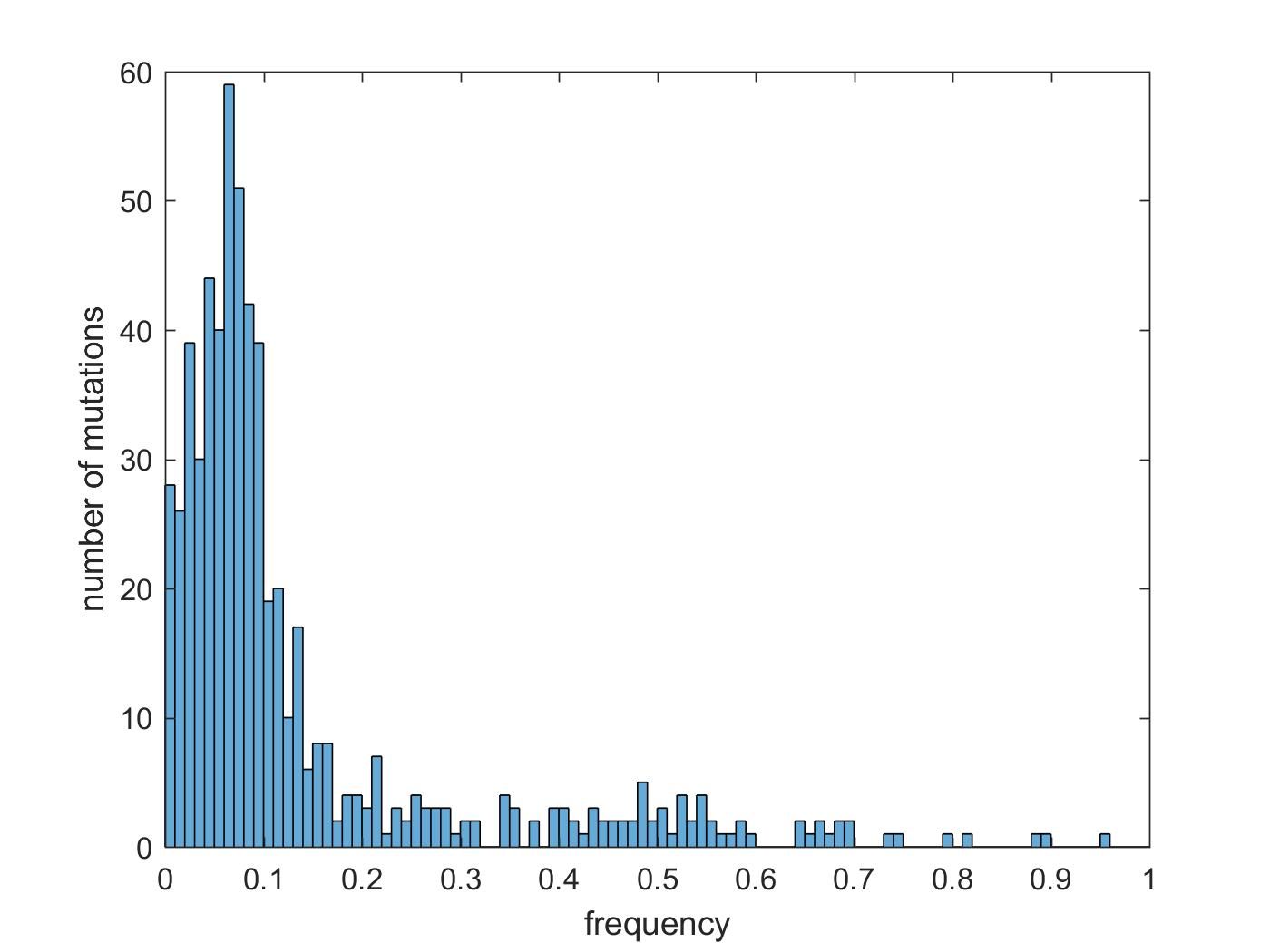}
\caption{A histogram of mutation frequencies from a lung adenocarcinoma.}\label{499017figure}

\includegraphics[scale=0.5]{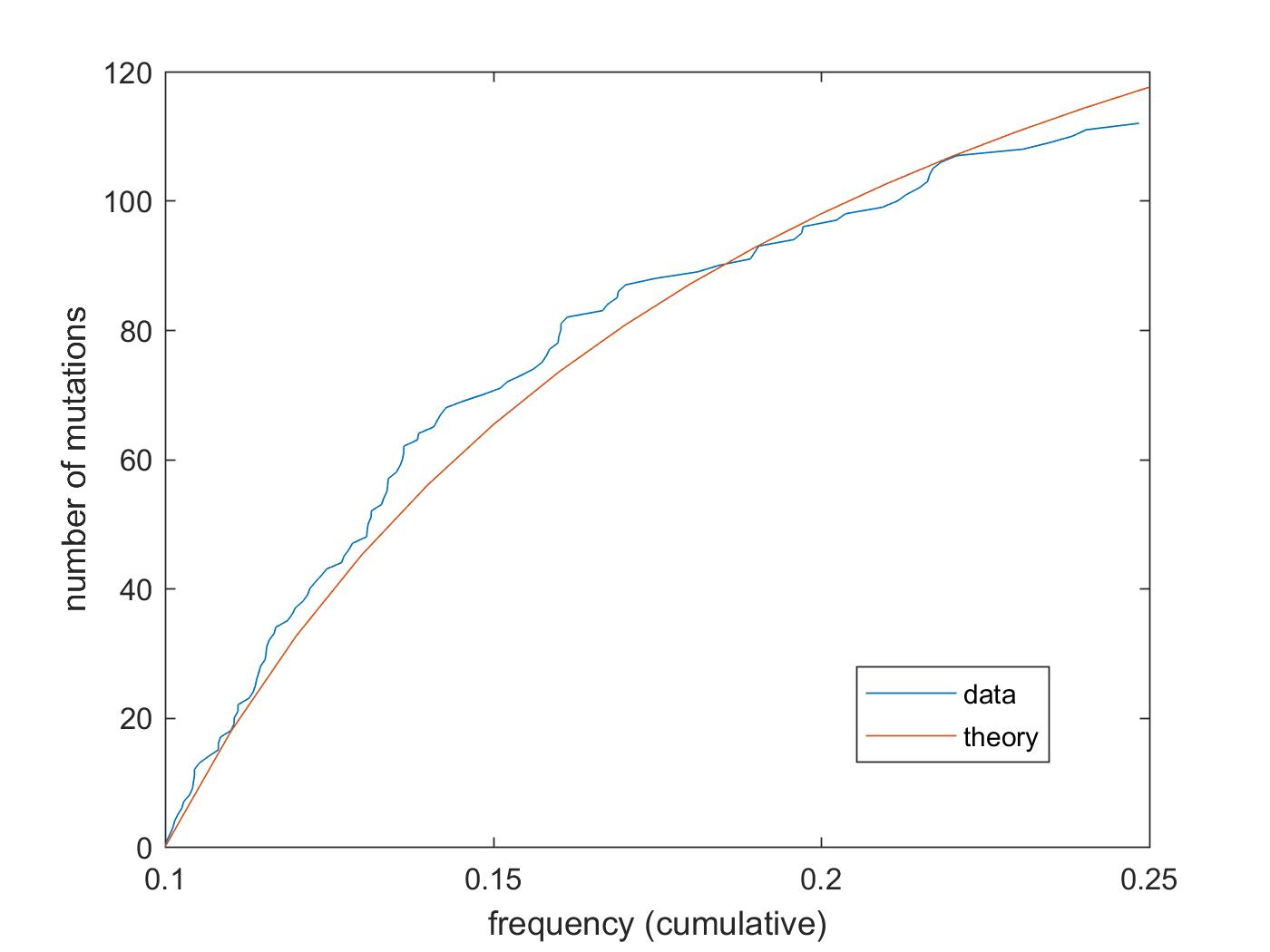}
\caption{The number of mutations (of the lung adenocarcinoma) whose frequency is in the interval $(0.1,x)$, for $x\in(0.1,0.25)$. }\label{datatheoryfigure}
\end{center}
\end{figure}

Our method to estimate mutation rates is, to a large extent, inspired by \cite{gs,ib}. Their attention is restricted to a subset of mutations. They ignore mutations at frequency less than $0.1$, saying that their detection is too unreliable. They ignore mutations above frequency 0.25, in order to neglect mutations present in the initial cell (which are relatively few). We do the same.

Write
\[
\mathcal{M}(a,b)=|\{j\in\{1,..,L\}:F_j\in(a,b)\}|
\]
for the number of mutations with frequency in $(a,b)\subset(0.1,0.25)$. Then, adapting Corollary \ref{lmfsfsm} to (\ref{mfmtd}), the expected number of mutations with frequency in $(a,b)$ is
\bq\label{enmfe}
\mathbb{E}\mathcal{M}(a,b)\approx \mu|\mathcal{S}|(a^{-1}-b^{-1}).
\eq
Under different models, \cite{gs,ib} derive the same approximation (\ref{enmfe}). They estimate the mutation rate $\mu$ by applying a linear regression to (\ref{enmfe}). We simplify matters even further. Our estimator for $\mu$ is
\bq\label{mresti}
\hat{\mu}=\frac{\mathcal{M}(0.1,0.25)}{6|\mathcal{S}|},
\eq
which (\ref{enmfe}) says is asymptotically unbiased. Now let's calculate $\hat{\mu}$ for the data example. The data shows mutations on the exome, which has rough size $|\mathcal{S}|=3\times10^8$ \cite{ib}. And the number of mutations in the specified frequency range is $\mathcal{M}(0.1,0.25)=112$. This gives
\[
\hat{\mu}=6.2\times10^{-8}.
\]
Next let's consider mutation rate heterogeneity. Write $\mu_\chi$ for the rate that nucleotide $\chi\in\mathcal{N}$ mutates. Partition the genetic sites:
\[
\mathcal{S}=\mathcal{S}_A\cup\mathcal{S}_C\cup\mathcal{S}_G\cup\mathcal{S}_T,
\]
where
\[
\mathcal{S}_\chi=\{i\in\mathcal{S}:u_i=\chi\}
\] is the set of sites which are represented by nucleotide $\chi$ in the initial cell. Just as before,
\[
\hat{\mu}_\chi=\frac{\mathcal{M}_\chi(0.1,0.25)}{6|\mathcal{S}_\chi|}
\]
is an unbiased estimator for $\mu_\chi$. The data gives

\[
(\hat{\mu}_A,\hat{\mu}_C,\hat{\mu}_G,\hat{\mu}_T)=(0.7,12.8,15.0,1.5)\times10^{-8}.
\]
This method could easily be extended to offer more detail, for example to estimate the rate at which nucleotide $A$ mutates to $C$ or to estimate mutation rates on different chromosomes.

The just presented statistical analysis is of course simple and brief. We recommend \cite{tavarerecent} for a far more comprehensive statistical analysis of mutation frequency data. However their infinite-sites framework does not consider mutation rate heterogeneity.

\begin{appendix}

\section*{Appendix}\label{cdca}
A heuristic `proof' of Conjecture \ref{lmfwsc} is given.

First we argue that, in the conjecture's limit, selection is unimportant. Write
\[
\mathcal{G}_\text{sel}=\{v\in\mathcal{G}:\exists i\in\mathcal{S}_\text{sel},v_i\not=u_i\}
\]
for the set of genomes which are mutated at a selective site. Write
\[
Q_\text{sel}^\mu(t)=\frac{\sum_{v\in\mathcal{G}_\text{sel}}X_v^\mu(t)}{\sum_{v\in\mathcal{G}}X_v^\mu(t)}
\]
for the proportion of cells at time $t\geq0$ whose genomes are mutated at a selective site. Then, according to Theorem \ref{mt},
\[
\left(Q_\text{sel}^\mu(\sigma_n^\mu)|\sigma_n^\mu<\infty\right)\rightarrow0
\]
in probability. Therefore we neglect selection.

Cells divide and die at rates $\alpha(u')$ and $\beta(u')$, which we now abbreviate to $\alpha$ and $\beta$. Some cells have an ultimately surviving lineage of descendants. Other cells eventually have no surviving descendants. Name these cells immortal and mortal respectively. In a supercritical birth-death branching process, it is well-known (eg. \cite{durr}) that the immortal cells grow as a Yule process and the mortal cells grow as a subcritical branching process. An immortal cell divides to produce two immortal cells at rate $\alpha-\beta$, or it divides to produce one immortal and one mortal cell at rate $2\beta$. A mortal cell divides at rate $\beta$ to produce two mortal cells, or it dies at rate $\alpha$. Because the process is conditioned to reach a large population size, let's assume that the initial cell is immortal.

The notation of Section \ref{sec:Yuletreesec}, $\mathcal{T}=\cup_{l=0}^\infty\{0,1\}^l$ and its partial ordering $\prec$, will be used to represent the immortal cells. Let $(A_x)_{x\in\mathcal{T}}$ be i.i.d. Exp($\alpha-\beta$) random variables, which represent the times for immortal cells to divide to produce two immortal cells. The immortal cells at time $t\geq0$ are
\[
\mathcal{T}_t=\left\{x\in\mathcal{T}:\sum_{y\prec x}A_y\leq t<\sum_{y\preceq x}A_y\right\}.
\]
The immortal descendants of $x\in\mathcal{T}$ are
\[
\mathcal{D}^I_x=\{y\in\mathcal{T}:x\preceq y\}.
\]
The number of immortal descendants of cell $x$ at time $t$ is
\[
D_{x,t}^I=|\mathcal{D}^I_x\cap\mathcal{T}_t|.
\] 
Let $\left((R_x(t))_{t\geq0}\right)_{x\in\mathcal{T}}$ be i.i.d. Poisson processes with rate $2\beta$. Write $R_{x,r}=\min\{t\geq0:R_x(t)=r\}$ for $r=1,..,R_x(A_x)$. Then the seeding times of mortal cells are
\[
S_{x,r}=\sum_{y\prec x}A_y+R_{x,r}.
\]
Each seeding event initiates a subpopulation of mortal cells; let $(Y_{x,r}(t))_{t\geq0}$ be i.i.d. birth-death branching processes with birth and death rates $\beta$ and $\alpha$. Then the number of mortal descendants of $x$ at time $t$ is
\[
D_{x,t}^M=\sum_{y\in\mathcal{D}_x}\sum_{r=1}^{R_y(A_y)}1_{\{t-S_{y,r}\geq0\}}Y_{y,r}\left(t-S_{y,r}\right).
\]
The number of descendants of $x$ at time $t$ is
\[
D_{x,t}=D_{x,t}^I+D_{x,t}^M.
\]
The next result shows the long-term proportion of a cell's descendants which are immortal. The result is a basic consequence of classic branching process theory \cite{an}, and was mentioned in its specific form by \cite{durr}.
\begin{manuallemma}{A.1}\label{cttbpr}There is $c\in(0,\infty)$ with
\[
\lim_{t\rightarrow\infty}\frac{D_{x,t}^I}{D_{x,t}}=c
\]
almost surely.
\end{manuallemma}
We use Lemma \ref{cttbpr} to see the number of descendants of a cell as a proportion of the total population.
\begin{manuallemma}{A.2}\label{wpocdfxwd}For $x\in\mathcal{T}\backslash\{\emptyset\}$,
\[
\lim_{t\rightarrow\infty}\frac{D_{x,t}}{D_{\emptyset,t}}=P_x
\]
almost surely, where the $P_x$ are as in Lemma \ref{btst}.
\begin{proof}
By Lemma \ref{btst} and Lemma \ref{cttbpr},
\[
\frac{D_{x,t}}{D_{\emptyset,t}}=\frac{D_{x,t}}{D_{x,t}^I}\frac{D_{x,t}^I}{D_{\emptyset,t}^I}\frac{D_{\emptyset,t}^I}{D_{\emptyset,t}}
\]
converges to the required limit.
\end{proof}
\end{manuallemma}
Let's look at mutations. In the proof of Theorem \ref{dslfm} it was shown that the number of new mutations to arise at a cell's birth is approximately Poisson. Here, with heterogeneous mutation rates, the number of new mutations to arise at a cell's birth is approximately Poisson with mean
\ba
\eta:=\sum_{j\in J}\sum_{\psi\in\mathcal{N}\backslash\{u(j)\}}\eta^{u(j),\psi}(j).
\ea
Each $x\in\mathcal{T}\backslash\{\emptyset\}$ witnesses $1+R_x(A_x)$ cell divisions, while $\emptyset$ witnesses $R_\emptyset(A_\emptyset)$ cell divisions (one less because there is not a cell division associated to the initiation of $\emptyset$). So the number of new mutations to arise at $x$ is
\ba
\begin{cases}\sum_{r=0}^{R_x(A_x)}M_{x,r},\quad&x\not=\emptyset;\\
\sum_{r=1}^{R_x(A_x)}M_{x,r},\quad&x=\emptyset;
\end{cases}
\ea
where the $M_{x,r}$ are i.i.d. Poisson random variables with mean $\eta$.  In the proof of Theorem \ref{dslfm} it was also shown that a mutation which arises in cell $x$ will have approximate frequency $P_x$. Here, thanks to Lemma \ref{wpocdfxwd}, the situation appears identical. It only remains to discuss mutations arising in mortal cells. Any subpopulation of cells which descended from a mortal cell must eventually die out. Hence mutations arising in mortal cells are negligible when compared to the infinite total population size.

\end{appendix}
\section*{Declarations of interest}None.
\section*{Acknowledgements}
We thank Michael Nicholson, Trevor Graham, and Marc Williams for inspiring discussions. We thank two anonymous referees for numerous helpful corrections and suggestions. David Cheek was supported by The Maxwell Institute Graduate School in Analysis and its Applications, a Centre for Doctoral Training funded by the UK Engineering and Physical Sciences Research Council (grant EP/L016508/01), the Scottish Funding Council, Heriot-Watt University and the University of Edinburgh.

\bibliographystyle{plain}

\bibliography{References} 
\end{document}